\documentclass[11pt,reqno]{amsart}
\usepackage[a4paper, hmargin={2.7cm,2.7cm},vmargin={3.3cm,3.3cm}]{geometry}
\usepackage[english]{babel}
\usepackage{color}
\usepackage[noadjust]{cite}
\usepackage{amsmath}
\usepackage{amssymb}
\usepackage{amsfonts}
\usepackage{amsthm}
\usepackage{enumerate}
\usepackage{amsthm}
\usepackage{dsfont}
\usepackage{todonotes}
\usepackage{etoolbox}
\usepackage{comment}

\numberwithin{equation}{section}

\makeatletter
\patchcmd{\ttlh@hang}{\parindent\z@}{\parindent\z@\leavevmode}{}{}
\patchcmd{\ttlh@hang}{\noindent}{}{}{}
\makeatother

\newcommand\numberthis{\addtocounter{equation}{1}\tag{\theequation}}

\theoremstyle{plain}
\newtheorem{theorem}{Theorem}[section]
\newtheorem{lemma}[theorem]{Lemma}
\newtheorem{proposition}[theorem]{Proposition}

\theoremstyle{definition}

\theoremstyle{remark}

\def\XXint#1#2#3{{\setbox0=\hbox{$#1{#2#3}{\int}$ }
\vcenter{\hbox{$#2#3$ }}\kern-.6\wd0}}

\DeclareMathOperator{\Span}{\overline{span}}

\DeclareMathOperator{\vol}{vol}
\DeclareMathOperator{\vN}{VN_{\sigma} (\Gamma)}

\DeclareMathOperator{\VN}{\vN}
\DeclareMathOperator{\id}{id}

\newcommand{\Hpi}{\mathcal{H}_{\pi}}
\newcommand{\Hil}{\mathcal{H}}

\newcommand{\cdim}{{\rm cdim}}
\newcommand{\ctr}{{\rm ctr}}

\newcommand{\cB}{\mathcal{B}}

\newcommand{\cZ}{\mathcal{Z}}

\title[Density conditions with stabilizers for lattice orbits of Bergman kernels]{Density conditions with stabilizers for lattice orbits of Bergman kernels on bounded symmetric domains}

\author{Martijn Caspers}

\author{Jordy Timo van Velthoven}

\address{Delft University of Technology,
Faculty EECMS/DIAM,
Mekelweg 4, Building 36,
2628 CD Delft, The Netherlands.}
\email{m.p.t.caspers@tudelft.nl}
\email{j.t.vanvelthoven@tudelft.nl}

\makeatletter
\@namedef{subjclassname@2020}{\textup{2020} Mathematics Subject Classification}
\makeatother



\begin{document}

\maketitle

\begin{abstract}
Let $\pi_{\alpha}$ be a holomorphic discrete series representation of a connected semi-simple Lie group $G$ with finite center, acting on a weighted Bergman space $A^2_{\alpha} (\Omega)$ on a bounded symmetric domain $\Omega$, of formal dimension $d_{\pi_{\alpha}} > 0$. It is shown that if the Bergman kernel $k^{(\alpha)}_z$ is a cyclic vector for the restriction $\pi_{\alpha} |_{\Gamma}$ to a lattice $\Gamma \leq G$ (resp. $(\pi_{\alpha} (\gamma) k^{(\alpha)}_z)_{\gamma \in \Gamma}$ is a frame for $A^2_{\alpha}(\Omega)$),
then $\vol(G/\Gamma) d_{\pi_{\alpha}} \leq |\Gamma_z|^{-1}$. The estimate $\vol(G/\Gamma) d_{\pi_{\alpha}} \geq |\Gamma_z|^{-1}$ holds for $k^{(\alpha)}_z$ being a $p_z$-separating vector (resp. $(\pi_{\alpha} (\gamma) k^{(\alpha)}_z)_{\gamma \in \Gamma / \Gamma_z}$ being a Riesz sequence in $A^2_{\alpha} (\Omega)$).
These estimates improve on general density theorems for restricted discrete series through the dependence on the stabilizers, while recovering in part sharp results for $G = \mathrm{PSU}(1, 1)$.
\end{abstract}

\section{Introduction}
Let $\Omega \subset \mathbb{C}^d$ be a bounded symmetric domain,
realized as a symmetric space $G/K$ of a semi-simple Lie group $G$
with trivial center and a maximal compact subgroup $K \leq G$.
For $\alpha > d' - 1$, with $d'$ being the genus of $\Omega$,
the associated weighted Bergman space $ A^2_{\alpha}  :=  A^2_{\alpha} (\Omega)$ is the closed subspace of holomorphic functions in the weighted $L^2$-space
\[
 L^2 (\Omega, \mu_{\alpha}) := \bigg\{ f : \Omega \to \mathbb{C} \; : \; \int_{\Omega} | f(w) |^2 \; d\mu_{\alpha} (w) < \infty  \bigg\},
\]
where $\mu_{\alpha}$ is a quasi-invariant probability measure on $\Omega$, cf. Section \ref{sec:holomorphic_discrete} for precise definitions.
It is a reproducing kernel Hilbert space, whose kernel will be denoted by
\[
 k^{(\alpha)}  : \Omega \times \Omega \to \mathbb{C}, \quad \text{and} \quad k^{(\alpha)}_z (w) :=  \overline{k^{(\alpha)} (z,w)}, \quad w,z \in \Omega.
\]
The group $G$ acts unitarily on $A^2_{\alpha} (\Omega)$ via the holomorphic discrete series, given by
\begin{align} \label{eq:jacobian_intro}
\pi_{\alpha} (g) f (w) = J_{g^{-1}} (w)^{\alpha / d'} f (g^{-1} \cdot w), \quad g \in G, \; w \in \Omega,
\end{align}
where $J_{g^{-1}}$ denotes the complex Jacobian of the action of $g^{-1}$ on $\Omega$. If $\alpha /d'$ is non-integral,
then the action $\pi_{\alpha}$ determines only a projective unitary representation on $A^2_{\alpha} (\Omega)$
due to the ambiguity of the power $J_{g^{-1}}( w)^{\alpha / d'}$ in \eqref{eq:jacobian_intro}. For any $\alpha > d' - 1$,
the (projective) representation $\pi_{\alpha}$ is irreducible and square-integrable; its formal dimension will be denoted by $d_{\pi_{\alpha}} > 0$.

\subsection{Density conditions} The orthogonality relations for discrete series $(\pi_{\alpha}, A^2_{\alpha})$ yield that
\begin{align} \label{eq:ortho_intro}
\int_G |\langle f, \pi_{\alpha} (x) \eta \rangle |^2 \; d\mu_G (x) = d_{\pi_{\alpha}}^{-1} \| f \|_{A^2_{\alpha}}^2 \| \eta \|_{A^2_{\alpha}}^2 \quad \text{for all} \quad f, \eta \in A^2_{\alpha} (\Omega).
\end{align}
In particular, the identity \eqref{eq:ortho_intro} implies that any orbit $\pi_{\alpha} (G) \eta$ of a non-zero $\eta \in A^2_{\alpha} (\Omega) $ is overcomplete, i.e., it contains proper subsystems that are still complete. Conditions on a discrete index set $\Lambda \subset G$ under which a subsystem $\pi_{\alpha}(\Lambda) \eta$ remains complete in $A^2_{\alpha} (\Omega)$ are generally referred to as \emph{density conditions} and have been studied in complex and harmonic analysis, mathematical physics and representation theory, see, e.g., \cite{bekka2004square, romero2020density, enstad2021density, seip1993beurling, kellylyth1999uniform, monastyrskij1974coherent, perelomov1972coherent, jones2020bergman}.

A prototypical example of a density condition is given by the following theorem, which is a special case of a general result
for unimodular groups, cf. \cite{kuhn1992restrictions, atiyah1976elliptic, atiyah1977geometric, goodman1989coxeter, bekka2004square}.

\begin{theorem}[Density theorem] \label{thm:density_intro}
Let $\Gamma \leq G$ be a lattice with (finite) co-volume $\vol(G/\Gamma)$.
\begin{enumerate}[(i)]
\item If there exists a cyclic vector $\eta \in A^2_{\alpha} (\Omega)$ for $\pi_{\alpha} |_{\Gamma}$, then $\vol(G/\Gamma) d_{\pi_{\alpha}} \leq 1$.
\item If there exists a separating vector $\eta \in A^2_{\alpha} (\Omega)$ for $\pi_{\alpha} |_{\Gamma}$, then $\vol(G/\Gamma) d_{\pi_{\alpha}} \geq 1$.
\end{enumerate}
(The quantity $\vol(G/\Gamma) d_{\pi_{\alpha}}$ is independent of the choice of Haar measure on $G$.)
\end{theorem}

Theorem \ref{thm:density_intro} is a well-known consequence of the coupling and dimension theory for von Neumann algebras and the corresponding \emph{Atiyah-Schmid formula} \cite{atiyah1976elliptic, atiyah1977geometric} for the (scalar-valued) dimension
\[
 \dim_{\VN} (A^2_{\alpha}) = \vol(G/\Gamma) d_{\pi_{\alpha}}
\]
of $A^2_{\alpha}(\Omega)$ as a Hilbert module over the (twisted) group von Neumann algebra $\VN$.

For an elementary proof of Theorem \ref{thm:density_intro} based on frame and representation theory, cf. \cite{romero2020density}.

The density conditions provided by Theorem \ref{thm:density_intro} are sharp in the sense that for any lattice $\Gamma \leq G$ satisfying $\vol(G/\Gamma) d_{\pi_{\alpha}} \leq 1$ (resp. $\vol(G/\Gamma) d_{\pi_{\alpha}} \geq 1$) there exists $\eta \in A^2_{\alpha} (\Omega)$ being a cyclic vector (resp. separating vector). Converse statements of this type can also be found in (or deduced from) the aforementioned papers \cite{kuhn1992restrictions, atiyah1976elliptic, atiyah1977geometric, goodman1989coxeter, bekka2004square, romero2020density}.

\subsection{Coherent states}
The density conditions provided by Theorem \ref{thm:density_intro} are generally not optimal for the cyclicity or separateness of one specific vector.
For example, for the group $G = \mathrm{PSU}(1,1) = \mathrm{SU}(1,1) / \{-I, I \}$ acting on the unit disk $\mathbb{D} \subset \mathbb{C}^1$ and $\eta = k^{(\alpha)}_z$ being the Bergman kernel function of a fixed $z \in \mathbb{D}$, the following density conditions involving the (finite) stabilizer $\Gamma_z := \{ \gamma \in \Gamma : \gamma \cdot z = z \}$ of a lattice $\Gamma \leq G$ were shown in \cite{perelomov1974coherent, kellylyth1999uniform}:
\[
\vol(G/\Gamma) d_{\pi_{\alpha}} < |\Gamma_z|^{-1} \quad \Longrightarrow \quad \Span \pi_{\alpha} (\Gamma) k_z^{(\alpha)} = A^2_{\alpha} (\mathbb{D}) \quad \Longrightarrow \quad \vol(G/\Gamma) d_{\pi_{\alpha}} \leq |\Gamma_z|^{-1}.
\]
In fact, it was recently shown in \cite{jones2020bergman} that $\vol(G/\Gamma) d_{\pi_{\alpha}} = |\Gamma_z|^{-1}$ is also sufficient for $\pi_{\alpha} (\Gamma) k^{(\alpha)}_z$ to be complete in $A^2_{\alpha} (\mathbb{D})$, so that the range $\vol(G/\Gamma) d_{\pi_{\alpha}} \leq |\Gamma_z|^{-1}$ provides a characterization of the completeness of the system $\pi_{\alpha} (\Gamma) k^{(\alpha)}_z$ in $A^2_{\alpha} (\mathbb{D})$.

An orbit $\pi_{\alpha} (G) k_{z}^{(\alpha)}$ of a Bergman kernel function $k_{z}^{(\alpha)}$ is sometimes referred to as a \emph{coherent state system}, see, e.g., \cite{monastyrskij1974coherent, perelomov1972coherent, perelomov1986generalized, neeb1996coherent, rawnsley1977coherent, lisiecki1990kaehler, perelomov1975coherent}.
The aim of this paper is to derive density conditions on a lattice $\Gamma \leq G$ for associated subsystems $\pi_{\alpha} (\Gamma) k_z^{(\alpha)}$ to be complete.

The first main result obtained is the following theorem.

\begin{theorem} \label{thm:main1_intro} Let $\Gamma \leq G$ be a lattice and let $\Gamma_z := \{ \gamma \in \Gamma : \gamma \cdot z = z \}$ be the (finite) stabilizer of a point $z \in \Omega$.
For arbitrary $\alpha > d' - 1$,
\begin{enumerate}[(i)]
\item
If $k^{(\alpha)}_z \in A^2_{\alpha} (\Omega)$ is a cyclic vector for $\pi_{\alpha} |_{\Gamma}$, then $\vol(G/\Gamma) d_{\pi_{\alpha}} \leq |\Gamma_z|^{-1}$.
\item If $k^{(\alpha)}_z \in A^2_{\alpha} (\Omega)$ is a $p_z$-separating vector for $\pi_{\alpha} |_{\Gamma}$, then $\vol(G/\Gamma) d_{\pi_{\alpha}} \geq |\Gamma_z|^{-1}$.
\end{enumerate}
\end{theorem}

The above mentioned results of \cite{perelomov1974coherent, kellylyth1999uniform, jones2020bergman} for $G = \mathrm{PSU}(1,1)$ show that assertion (i) of Theorem \ref{thm:main1_intro} is sharp, in the sense that the upper bound $|\Gamma_z|^{-1}$ for $\vol(G/\Gamma) d_{\pi_{\alpha}}$ cannot be replaced by a smaller value. In addition, this estimate is stronger than the automorphic weight bound for the completeness of coherent state subsystems obtained in \cite{monastyrskij1974coherent}, which is only valid conditional the existence of adequate modular forms
and is weaker than the density condition provided by Theorem \ref{thm:density_intro} (hence, the one of Theorem \ref{thm:main1_intro}) in the case of $G = \mathrm{PSU}(1,1)$, cf. \cite[Lemma 5.3]{kellylyth1999uniform} and \cite[Section 9.1.1]{romero2020density} for a detailed discussion.

In assertion (ii) of Theorem \ref{thm:main1_intro}, the vector $k_z^{(\alpha)}$ being a $p_z$-separating vector for $\pi_{\alpha} |_{\Gamma}$ means that it is a separating vector for the reduced von Neumann algebra $p_z \pi_{\alpha} (\Gamma)'' p_z$ for a projection $p_z \in \pi_{\alpha} (\Gamma)''$ determined by the stabilizer $\Gamma_z$. In the particular case of a trivial stabilizer $\Gamma_z$, the projection $p_z$ equals the identity operator $\id_{\Hpi}$, and assertion (ii) recovers the corresponding estimate of Theorem \ref{thm:main1_intro}.
The use of $p_z$-separating vectors is an essential ingredient in the main result on Riesz sequences (see Theorem \ref{thm:main2_intro}). Assertion (ii) seems to be new even for the special case of $G = \mathrm{PSU}(1,1)$.

For the proof of Theorem \ref{thm:main1_intro}, the coupling and dimension theory for von Neumann algebras is used. In particular, the technique for obtaining part (i) resembles the recent approach of \cite{jones2020bergman} for $G = \mathrm{PSU}(1,1)$ in which the von Neumann dimension $\dim_{\VN}(\VN k_z^{(\alpha)})$ of the Hilbert submodule $\VN k_z^{(\alpha)}$ is compared to the von Neumann dimension $\dim_{\VN} (A^2_{\alpha})$ of the Bergman space $A^2_{\alpha} (\Omega)$ itself.
This comparison is the key to improve the conditions of Theorem \ref{thm:density_intro} for one specific vector. It should be mentioned that for obtaining part (i) of Theorem \ref{thm:main1_intro} several technical modifications are required compared to \cite{jones2020bergman}, where certain ad hoc arguments for $\mathrm{PSU}(1,1)$ are used, such as cyclicity of the stabilizer subgroups.
The proof of assertion (ii) follows an approach dual to that of assertion (i).

\subsection{Frames and Riesz sequences} While the notions of cyclic and separating vectors are common in operator algebras and representation theory, the (stronger) notions of frames (resp. sampling sets) and Riesz sequences (resp. interpolation sets) are more common in complex and harmonic analysis, see, e.g., \cite{seip2004interpolation, young2001introduction}.  The system $\pi_{\alpha} (\Gamma) k_{z}^{(\alpha)}$ is called a
\emph{frame} for $A^2_{\alpha} (\Omega)$ if
\begin{align} \label{eq:frame_intro}
 A \| f \|_{A^2_{\alpha}}^2 \leq \sum_{\gamma \in \Gamma} |\langle f, \pi_{\alpha} (\gamma) k_z^{(\alpha)} \rangle |^2 \leq B \| f \|_{A^2_{\alpha}}^2 \quad \text{for all} \quad f \in A^2_{\alpha} (\Omega)
\end{align}
for some constants $0<A\leq B<\infty$, and it is called a \emph{Riesz sequence} in $A^2_{\alpha} (\Omega)$ if
\begin{align} \label{eq:riesz_intro}
 A \| c \|_{\ell^2}^2 \leq \bigg\| \sum_{\gamma \in \Gamma} c_{\gamma} \pi_{\alpha} (\gamma) k_z^{(\alpha)} \bigg\|_{A^2_{\alpha}}^2 \leq B \| c \|_{\ell^2}^2 \quad \text{for all} \quad c = (c_{\gamma} )_{\gamma \in \Gamma} \in \ell^2 (\Gamma).
\end{align}
The lower bound in \eqref{eq:frame_intro} (resp. \eqref{eq:riesz_intro}) implies, in particular, that a frame (resp. Riesz sequence) $\pi_{\alpha} (\Gamma) k_{z}^{(\alpha)}$ is complete (resp. linear independent) in $A^2_{\alpha} (\Omega)$.
For the lattice orbit $\pi_{\alpha} (\Gamma) k_{z}^{(\alpha)}$ to be a Riesz sequence,
it is therefore necessary that the stabilizer $\Gamma_z$ is trivial; otherwise, only the reduced system $(\pi_{\alpha} (\gamma) k_{z}^{(\alpha)})_{\gamma \in \Lambda_z}$, where $\Lambda_z$ is a set of representatives of $\Gamma / \Gamma_z$, can be expected to satisfy the corresponding Riesz inequalities \eqref{eq:riesz_intro}.

The following result is obtained as a consequence of Theorem \ref{thm:main1_intro}.

\begin{theorem} \label{thm:main2_intro} Let $\Gamma \leq G$ be a lattice and let $\Gamma_z := \{ \gamma \in \Gamma : \gamma \cdot z = z \}$ be the (finite) stabilizer of a point $z \in \Omega$. Let $\Lambda_z \subseteq \Gamma$ be a set of representatives of $\Gamma / \Gamma_z$.
For arbitrary $\alpha > d' - 1$,

\begin{enumerate}[(i)]
\item
If $(\pi_{\alpha} (\gamma) k^{(\alpha)}_z)_{\gamma \in \Gamma}$ is a frame for $A^2_{\alpha} (\Omega)$, then $\vol(G/\Gamma) d_{\pi_{\alpha}} \leq |\Gamma_z|^{-1}$.
\item If $(\pi_{\alpha} (\gamma) k^{(\alpha)}_z)_{\gamma \in \Lambda_z}$ is a Riesz sequence in $A^2_{\alpha} (\Omega)$, then $\vol(G/\Gamma) d_{\pi_{\alpha}} \geq |\Gamma_z|^{-1}$.
\end{enumerate}
\end{theorem}

Part (i) of Theorem \ref{thm:main2_intro} is an immediate consequence of Theorem \ref{thm:main1_intro}. For part (ii),
it is shown that the Riesz property of $(\pi_{\alpha} (\gamma) k^{(\alpha)}_z)_{\gamma \in \Lambda_z}$ implies that $k^{(\alpha)}_z$ is a $p_z$-separating vector. This implication is non-trivial and was shown for trivial stabilizers in \cite{romero2020density}. Proposition \ref{prop:riesz_separating} is an extension of \cite[Proposition 5.2]{romero2020density} to non-trivial stabilizers and is  a key result of this paper.

For $G = \mathrm{PSU}(1,1)$ acting on the unit disk $\mathbb{D} \subset \mathbb{C}^1$, the papers \cite{seip1993beurling, kellylyth1999uniform} provide a complete characterization of the frame and Riesz property of systems $(\pi_{\alpha} (\gamma) k^{(\alpha)}_z)_{\gamma \in \Lambda_z}$ in $A^2_{\alpha} (\mathbb{D})$; see \cite[Section 9.1]{romero2020density} for a detailed discussion. In particular, the results in \cite{seip1993beurling, kellylyth1999uniform} show that the density conditions of Theorem \ref{thm:main2_intro} are \emph{strict} inequalities. The techniques used in the present paper seem not to be able to yield strict versions of Theorem \ref{thm:main2_intro}. Lastly, it is expected that complementing sufficient conditions to Theorem \ref{thm:main2_intro} do not hold for domains $\Omega \subset \mathbb{C}^d$ in dimension $d > 1$; see \cite{marco2003density, jevtic1996interpolating} for some possible generalizations of the results \cite{seip1993beurling} to the unit ball.
For the failure of sufficiency of the lattice density conditions \cite{bekka2004square, romero2020density} underlying Theorem \ref{thm:density_intro} for coherent states in the Bargmann-Fock spaces on $\mathbb{C}^d$ with $d > 1$, see, e.g.,  \cite{Massaneda2000interpolating, ortega-cerda2006interpolation}.

\subsection{Extensions} For the sake of simplicity and comparison, the results in the introduction have only been stated for connected semi-simple Lie groups $G$ with a trivial center. However, the results in the main part of this paper are also valid for groups having merely a finite center. A technical difficulty that arises when $G$ has a non-trivial center is that the group von Neumann algebra associated with a lattice $\Gamma \leq G$ might no longer be a factor, i.e., it has a non-trivial center. In particular, this requires the use of center-valued rather than scalar-valued traces and von Neumann dimensions.

In addition to Theorem \ref{thm:main1_intro} and Theorem \ref{thm:main2_intro}, this paper also provides new density conditions
for general discrete series (i.e., possibly non-holomorphic discrete series) of semi-simple Lie groups in Theorem \ref{thm:density_main1} and Theorem \ref{thm:main1_frame}.

\subsection*{Organization} Section \ref{sec:vN} provides general results on von Neumann algebras and associated Hilbert modules that will be used for the main results of this paper. Several preliminary results on discrete series representations and their restrictions to lattices will be presented in Section \ref{sec:discrete_lattice}. These results are stated for general unimodular groups and may be of independent interest.
Section \ref{sec:discrete_semi} concerns results on discrete series and lattices that are specific for semi-simple Lie groups. Section \ref{sec:holomorphic_discrete} is devoted to the holomorphic discrete series representations and the proof of Theorem \ref{thm:main1_intro} and Theorem \ref{thm:main2_intro}.

\section{Hilbert modules associated to von Neumann algebras} \label{sec:vN}
This section concerns results on von Neumann algebras and Hilbert modules that will be exploited in subsequent sections. For the used background on these topics, see, e.g., \cite{takesaki2002theory, dixmier1981von, jones1997subfactors}.

\subsection{GNS construction}\label{Sect=GNS}
Let $B(\mathcal{H})$ be the algebra of bounded operators on a Hilbert space $\Hil$.
A unital $\ast$-subalgebra of $\mathcal{H}$ that is closed in the strong operator topology (SOT) is called a \emph{von Neumann algebra}.  We denote $M' := \{ a \in B(\Hil) : a b = b a, \; \forall b \in M \}$ for the commutant of $M$.
A map between von Neumann algebras is called \emph{normal} if it is strongly continuous on the unit ball.
A von Neumann algebra is called \emph{finite} if it admits a faithful normal tracial state, i.e., a normal linear functional $\tau : M \to \mathbb{C}$ satisfying
$\tau(a^* a) > 0$ for all $a \in M \setminus \{0\}$ and $\tau(ab)=\tau(ba)$ for all $a,b \in M$.
Given such a trace $\tau$, the pair $(M, \tau)$ is sometimes also called a \emph{tracial} von Neumann algebra.

For a tracial von Neumann algebra $(M, \tau)$,
an associated inner product is defined by
\[
 \langle a, b \rangle_{\tau} = \tau(b^* a), \quad a,b \in M,
\]
and the corresponding norm is $\| a \|_{2, \tau} = \tau(a^* a)^{1/2}$.
The Hilbert space completion of $M$ with respect to $\langle \cdot , \cdot \rangle_{\tau}$ is denoted by $L^2 (M)$.
The GNS representation allows to identify $M$ with the subalgebra of $B(L^2(M))$ of all operators given by left multiplication with $M$.

\subsection{Center-valued trace}
The \emph{center} of a von Neumann algebra $M$ is the intersection $\cZ(M) := M \cap M'$.
Given a tracial von Neumann algebra $(M, \tau)$, there exists a unique central-valued trace $\ctr$ on $M$, i.e., a normal bounded linear map $\ctr : M \to \cZ(M)$ such that $\ctr(ab) = \ctr(ba)$ for all $a,b \in M$, $\ctr(ba) = b \ctr(a)$ for all $a \in M$ and $b \in \cZ(M)$, $\ctr(a) = a$ for all $a \in \cZ(M)$, and $\ctr(a^* a) = 0$ implies $a = 0$ for all $a \in M$.
If $M$ is a factor, i.e., $\cZ(M) = \mathbb{C} \id$, then $\ctr = \tau$.

\subsection{Hilbert modules} Let $(M, \tau)$ be a tracial von Neumann algebra.
A (left) \emph{Hilbert $M$-module} is a Hilbert space $\Hil$ together with a normal unital $*$-homomorphism $\rho : M \to B(\Hil)$. A Hilbert $M$-module $\Hil$ is called \emph{faithful} if $\rho$ is injective.
  In general, we shall suppress $\rho$ in the notation and denote the action of $M$ on $\mathcal{H}$ by $a \eta = \rho(a) \eta$.
The Hilbert space $L^2(M)$ forms an $M$-module, sometimes called the \emph{standard $M$-module}.

The direct sum of countable many copies of $L^2 (M)$ is a left $M$-module,
and will be denoted by $ L^2 (M) \otimes \ell^2 (\mathbb{N})$. Any separable Hilbert $M$-module $\Hil$ is isomorphic to a submodule of $L^2 (M) \otimes \ell^2 (\mathbb{N})$, see, e.g.,
\cite[Proposition 2.1.2]{jones1997subfactors}.

\subsection{Cyclic and separating vectors} Let $\Hil$ be a Hilbert $M$-module. A vector $\eta \in \Hil$ is called \emph{separating} if $a \eta = 0$ with $a \in M$ implies $a = 0$, and $\eta$ is \emph{cyclic} if $M \eta$ is dense in $\Hil$. Note that separating vectors can only exist if $\Hil$ is faithful.

  For a projection $p \in M$, a vector $\eta \in p \Hil$ is called \emph{$p$-separating} if  $a \in Mp$ and $a \eta = 0$ implies that $a = 0$. Note that if $p = \id$, then a $p$-separating vector is just a separating vector for $M$. In general, it can be shown that a vector $\eta$ is $p$-separating if and only if $\eta$ is separating for the reduced von Neumann algebra $pMp$.  Indeed, if $\eta$ is separating for $pMp$ and $a p \eta = 0$ for $a \in M$,  then for all partial isometries $u \in M$ with $uu^\ast \leq p$ we have  $uap \eta = 0$ and hence $uap =0$. This suffices to conclude that $qa p = 0$, where $q$ is the central support of $p$ (see, e.g., \cite{takesaki2002theory}). But $qap = aqp = ap$ and so $ap=0$.

\subsection{Von Neumann dimension}
Henceforth, assume $M$ is a finite von Neumann algebra.
Let $\Hil$ be a separable Hilbert $M$-module, viewed as a submodule of $L^2 (M) \otimes \ell^2 (\mathbb{N})$, and let $p$ be the corresponding projection of $L^2 (M) \otimes \ell^2 (\mathbb{N})$ onto $\Hil$. Then $p$ is in the commutant $M' \otimes B(\ell^2(\mathbb{N}))$ of $M$ in $B(L^2 (M) \otimes \ell^2 (\mathbb{N}))$.
The commutant $M' \otimes B(\ell^2(\mathbb{N}))$ is semi-finite, hence admits a normal, semi-finite center-valued trace $\Phi$, see, e.g., \cite[Theorem II.3.4]{takesaki2002theory}.
The \emph{center-valued von Neumann dimension} of $\Hil$ is defined by
\[
 \cdim_M (\Hil) = \Phi (p).
\]
Its definition is independent of the choice of projection $p$, cf. \cite{bekka2004square} for more details.

Let $\Hil_1$ and $\Hil_2$ be Hilbert $M$-modules (considered as submodules of $L^2 (M) \otimes \ell^2 (\mathbb{N})$), with corresponding projections
$p_1$ and $p_2$, respectively. If $\Hil_1$ is isomorphic to a submodule of $\Hil_2$, then $\cdim_M (\Hil_1) \leq \cdim_M (\Hil_2)$, and if $\Hil_1$ is isomorphic to $\Hil_2$, then $\cdim_M (\Hil_1) = \cdim_M (\Hil_2)$. If $p_1$ and $p_2$ are finite projections,
then $\Hil_1$ is isomorphic to a submodule of $\Hil_2$ if and only if $\cdim_M(\Hil_1) \leq \cdim_M(\Hil_2)$, and
$\cdim_M(\Hil_1) = \cdim_M(\Hil_2)$ if and only if $\Hil_1$ and $\Hil_2$ are isomorphic as $M$-modules.
For proofs, see \cite[Proposition 3.2.4]{goodman1989coxeter}, \cite[Proposition 2]{bekka2004square} and \cite[Proposition 2.3]{enstad2021density}

The center-valued dimension $\cdim_M (\Hil)$ coincides with the so-called \emph{coupling operator} for $M p$, see, e.g., \cite[Proposition 3.2.5]{goodman1989coxeter} and \cite[Proposition 1]{bekka2004square}.
The following fundamental property of the coupling operator (center-valued von Neumann dimension) is well-known, among others, it can be found in \cite[Proposition V.3.13]{takesaki2002theory}.

\begin{theorem} \label{thm:coupling}
 Let $M$ be a finite von Neumann algebra and let $\Hil$ be a separable $M$-module.
 \begin{enumerate}
  \item[(i)] $M$ admits a cyclic vector if, and only if, $\cdim_M (\Hil) \leq \id$.
  \item[(ii)] $M$ admits a separating vector if, and only if, $\cdim_M (\Hil) \geq \id$.
 \end{enumerate}
 Moreover, if the equivalent statements in (ii) hold, then $\Hil$ is faithful.
\end{theorem}

 Define the anti-linear isometry $J: L^2(M) \rightarrow L^2(M)$ determined by $a \mapsto a^\ast, \; a \in M$. $J$ is known as the \emph{modular conjugation}.  For $a \in M$ and $\eta \in L^2(M)$ we set $\eta a = J a^\ast J \eta$.
The estimates provided by the following two results will be essential in the sequel.

\begin{proposition}\label{Prop=DimEstimate}
 Let $M$ be a finite von Neumann algebra and let $\Hil$ be a separable $M$-module.
Let $\eta \in \Hil$ and let $p \in M$ be such that $p \eta = \eta$. Then
\[
\cdim_M(   \overline{ M \eta }  ) \leq \ctr(p).
\]
\end{proposition}
\begin{proof}
 Without loss of generality, it may be assumed that $\Hil = \overline{ M \eta }$ and in particular that $\Hil$ contains a cyclic vector. Therefore we may identify $\Hil$ with a sub-$M$-module of $L^2(M)$.  Let $q \in M$ be the  support projection of $\eta \in L^2(M)$; $q \in M$ is the projection onto $\overline{ \eta M}$. As $p \eta = \eta$ we have $q \leq p$ and so $\ctr(q) \leq \ctr(p)$. A direct calculation gives
 \[ L^2(M) q = J q J L^2(M) = J q L^2(M) = J \overline{\eta M} = \overline{ M \eta } = \Hil. \] Hence, $\cdim_M(   \Hil  ) = \Phi(JqJ) = \ctr(q)$, which concludes the proof.
\end{proof}

\begin{proposition} \label{prop:dimension_pseparating}
 Let $M$ be a finite von Neumann algebra and let $\Hil$ be a separable $M$-module.
 If there exists a $p$-separating vector in $p\Hil$, then $\cdim_M (\Hil) \geq \ctr(p)$.
\end{proposition}
\begin{proof}
 If $p\Hil$ admits a $p$-separating vector, then that vector is clearly separating for $pMp$. This implies that $\cdim_{pMp} (p\Hil) \geq \id$ by Theorem \ref{thm:coupling}. The result follows therefore directly from the well-known identity
 \[
  \ctr(p) \cdim_{pMp} (p\Hil) = \cdim_M (\Hil),
 \]
 see, e.g., \cite[Proposition 3.2.5(h)]{goodman1989coxeter} (or \cite[Section V]{takesaki2002theory}, \cite[Section III.6]{dixmier1981von}).
\end{proof}

\section{Discrete series representations restricted to lattices} \label{sec:discrete_lattice}
Throughout, $G$ denotes a second countable unimodular group with a fixed left Haar measure $\mu_G$. The identity element in $G$ is denoted by $e_G$.

\subsection{Discrete series and cocycles}
A projective unitary representation $(\pi, \Hpi)$ of $G$ on a (complex) Hilbert space $\Hpi$ is a map $\pi : G \to \mathcal{U}(\Hpi)$ such that $\pi(e_G) = \id_{\Hpi}$ and
\begin{align} \label{eq:cocycle_rep}
 \pi(x) \pi(y) = \sigma(x,y) \pi(xy), \quad x,y \in G,
\end{align}
for some $\sigma : G \times G \to \mathbb{T}$. The function $\sigma$ can be shown to satisfy the identities
\begin{align} \label{eq:cocycle_identity}
 \sigma(x, e_G) = \sigma(e_G, x) = 1 \quad \text{and} \quad
 \sigma(x,y) \sigma(xy,z) = \sigma(x, yz) \sigma(y,z), \quad x,y,z \in G.
\end{align}
Any projective representation will be assumed to be measurable,
i.e., the map $x \mapsto \langle f, \pi(x) g \rangle$ is Borel measurable for all $f, g \in \Hpi$.
This implies, in particular, that the function $\sigma$ appearing in \eqref{eq:cocycle_rep} is Borel measurable on $G \times G$.
It is called a \emph{cocycle} and
a projective unitary representation with cocycle $\sigma$ will simply be called a
\emph{$\sigma$-representation}. A cocycle $\sigma$ is called a \emph{coboundary} if there exists $u : G \to \mathbb{T}$ such that
\[
 \sigma(x,y) = u(x) u(y) \overline{u(xy)}, \quad x,y \in G.
\]

An element $x \in G$ is called $\sigma${\it-regular} in $G$ if $\sigma(x,y) = \sigma(y,x)$ for all $y \in Z(x)$, where $Z(x) := \{ y \in G : yx = xy \}$
is the centralizer of $x$ in $G$.
If $x$ is $\sigma$-regular, then $y x y^{-1}$ is $\sigma$-regular for all $y \in G$,
and the associated conjugacy class $C_x := \{ y x y^{-1} : y \in G\}$
is said to be \emph{$\sigma$-regular}.

A $\sigma$-representation $(\pi, \Hpi)$ is \emph{square-integrable} if there exists $\eta \in \Hpi \setminus \{0\}$ such that
$\int_G |\langle \eta, \pi(x) \eta \rangle|^2\; d\mu_G (x) < \infty$. For an irreducible, square-integrable $\pi$, there exists a unique $d_{\pi} > 0$, called the \emph{formal dimension} of $\pi$, such that
\begin{align} \label{eq:ortho}
\int_G | \langle f, \pi(x) \eta \rangle |^2 \; d\mu_G (x) = d_{\pi}^{-1} \| f \|_{\Hpi}^2 \| \eta \|_{\Hpi}^2
\end{align}
for all $f, \eta \in \Hpi$. An irreducible, square-integrable $\pi$ is called a \emph{discrete series $\sigma$-representation}.

For more background on (projective) discrete series representations, see, e.g., \cite{radulescu1998berezin, neeb1999holomorphy}.

\subsection{Lattices}
A discrete subgroup $\Gamma \leq G$ is called a \emph{lattice}
if $G/\Gamma$ admits a finite $G$-invariant Radon measure.
For a lattice $\Gamma \leq G$ and cocycle $\sigma : \Gamma \times \Gamma \to \mathbb{T}$,
the associated twisted left-regular representation $(\lambda_{\Gamma}^{\sigma}, \ell^2 (\Gamma))$ of $\Gamma$ is given by
\[
(\lambda_{\Gamma}^{\sigma} (\gamma) c )_{\gamma'} = \sigma(\gamma, \gamma^{-1} \gamma') c_{\gamma^{-1} \gamma'}, \quad \gamma, \gamma' \in \Gamma.
\]
The representations $(\pi|_{\Gamma}, \Hpi)$ and $(\lambda_{\Gamma}^{\sigma}, \ell^2 (\Gamma))$ satisfy the intertwining property
\[
C_{\eta} (\pi(\gamma) f) (\gamma') = \sigma(\gamma, \gamma^{-1} \gamma') C_{\eta} f (\gamma^{-1} \gamma') = \lambda_{\Gamma}^{\sigma} (\gamma) C_{\eta} f (\gamma'), \quad \gamma, \gamma' \in \Gamma,
\]
where $C_{\eta} f(\gamma') = \langle f, \pi(\gamma') \eta \rangle$  for a vector $\eta \in \Hpi$ such that
$ C_{\eta} f \in \ell^2 (\Gamma)$.

A \emph{Bessel vector} is an element $\eta \in \Hpi$ such that the associated map
\[
C_{\eta} : \Hpi \to \ell^2 (\Gamma), \quad f \mapsto (\langle f, \pi(\gamma) \eta \rangle)_{\gamma \in \Gamma}
\]
is bounded.
The space $\cB_{\pi}$ of Bessel vectors is norm dense in $\Hpi$, see, e.g., \cite[Lemma 7.1]{romero2020density}.

The restriction $\pi|_{\Gamma}$ and associated $\lambda_{\Gamma}^{\sigma}$
are generally related as follows:

\begin{lemma} \label{lem:intertwiner}
 Let $(\pi, \Hpi)$ be a discrete series $\sigma$-representation of $G$ and let $\Gamma \leq G$ be a lattice.
 There exists an index set $I$ and an isometry $U : \Hpi \to \oplus_{i \in I} \ell^2 (\Gamma)$
 interwining $\pi|_{\Gamma}$ and $\oplus_{i \in I} \lambda_{\Gamma}^{\sigma}$, i.e.,
 \[
  U \pi(\gamma)  =  \big( \oplus_{i \in I} \lambda_{\Gamma}^{\sigma} \big) (\gamma) U, \quad \gamma \in \Gamma.
 \]
\end{lemma}
\begin{proof}
Let $\mathcal{I}$ be the set of all pairs $(\mathcal{H}, U)$ with $\mathcal{H}$ a closed $\pi|_{\Gamma}$-invariant subspace of $\Hpi$ and $U: \mathcal{H} \rightarrow \oplus_{i \in I} \ell^2(\Gamma)$ an isometry into a direct sum over copies of $\ell^2 (\Gamma)$ indexed by some set $I = I(\mathcal{H}, U)$.  Then $\mathcal{I}$ carries the order $(\mathcal{H}, U) \leq (\mathcal{K}, V)$ if $\mathcal{H} \subseteq \mathcal{K}, I(\mathcal{H}, U) \subseteq I(\mathcal{K}, V)$ and $V$ extends $U$. Every linear chain $(\mathcal{H}_i, U_i)$ in $\mathcal{I}$ has an upper bounded $(\overline{ \oplus_i \Hil_i}, U)$ with $U$ the unique map such that $U\vert_{\Hil_i} = U_i$. By Zorn's lemma, $\mathcal{I}$ must contain a maximal element $I = (\mathcal{K}, V)$. For completing the proof, it remains therefore to show that $\mathcal{K} = \Hpi$ and that $\mathcal{I}$ is non-empty. The next paragraph shows both.

Assume towards a contradiction that $\mathcal{K}^\perp$ is a non-zero $\pi|_{\Gamma}$-invariant subspace of $\Hpi$. Let $p: \Hpi \rightarrow \mathcal{K}^\perp$ be the orthogonal projection onto $\mathcal{K}^\perp$, so that $p \in \pi(\Gamma)'$.  Since the Bessel vectors $\cB_{\pi}$ are norm dense in $\Hpi$, there exists  $\eta \in \cB_\pi$ such that $p \eta \not = 0$. Clearly, $p \eta \in \cB_\pi$, i.e., $C_{p \eta}$ is bounded from $\Hpi$ into $\ell^2(\Gamma)$. Furthermore, the intertwining property $\lambda_{\Gamma}^{\sigma} (\gamma) C_{p \eta} = C_{p \eta}  \pi(\gamma)$ holds for all  $\gamma \in \Gamma$, showing that the kernel   $\ker(C_{p \eta})$ is $\pi|_{\Gamma}$-invariant. Since $\mathcal{K} \subseteq \ker (C_{p \eta})$ by $\pi|_{\Gamma}$-invariance of $\mathcal{K}$, it follows that $\ker(C_{p \eta})^\perp \subseteq \mathcal{K}^\perp$.  Let $C_{p \eta} = U \vert C_{p \eta} \vert$ be the polar decomposition. Then $U$ forms a partial isometry from $\mathcal{K}^\perp$ into  $\ell^2(\Gamma)$ with support  $\ker(C_{p \eta})^\perp$ intertwining $\pi|_{\Gamma}$ and $\lambda_{\Gamma}^{\sigma}$. Therefore, $(\mathcal{K} \oplus \ker(U)^\perp, V \oplus U)$ is larger than $(\mathcal{K}, V)$, which contradicts the maximality of $(\mathcal{K}, V)$.
  \end{proof}

\subsection{Twisted von Neumann algebras}
For a discrete series $\sigma$-representation $(\pi, \Hpi)$ of $G$ and a lattice $\Gamma \leq G$, the associated von Neumann algebra is given by  $\pi(\Gamma) '' = (\pi(\Gamma)')'$. Similarly, the von Neumann algebra generated by $\lambda_{\Gamma}^{\sigma} (\Gamma) \subset \mathcal{B}(\ell^2(\Gamma))$ is given by $\VN := \lambda_{\Gamma}^{\sigma} (\Gamma)''$. It is called the (twisted)
\emph{group von Neumann algebra}. A faithful normal trace on $\VN$ is given by
\[
 \tau(a) = \langle a \delta_e, \delta_e \rangle, \quad a \in \VN,
\]
where $\delta_e \in \ell^2 (\Gamma)$ denotes the indicator at the identity in $\Gamma$. The GNS construction $L^2 (\VN)$ relative to $\tau$ is canonically isomorphic to $\ell^2 (\Gamma)$.

It should be mentioned that for a unitary $\sigma$-representation $\pi$ of $G$, its restriction $\pi|_{\Gamma}$ to a lattice $\Gamma \leq G$ could be (equivalent to) a genuine unitary representation, so that $\sigma : \Gamma \times \Gamma \to \mathbb{T}$ is trivial. The algebra $\VN$ is then an ordinary (non-twisted) group von Neumann algebra. However, for the purposes of the present paper, it is essential to treat possible twisted group von Neumann algebas, see, e.g., \cite[Appendix A]{jones2020bergman} for relevant examples.

The following simple lemma will be used repeatedly below.

\begin{lemma}\label{lem:projection}
Let $\Lambda \subset \Gamma$ be a finite group. Suppose that $\sigma(\gamma, \gamma') = u(\gamma) u(\gamma') \overline{u(\gamma \gamma') }$, where $\gamma, \gamma' \in \Lambda$, for some $u: \Lambda \rightarrow \mathbb{T}$. Then
\[
p := \frac{1}{ \vert \Lambda \vert} \sum_{  \gamma \in \Lambda  } \overline{u(\gamma)} \pi(\gamma)
\]
is a projection in $\pi(\Gamma)'' \subseteq \mathcal{B} (\Hpi)$.
\end{lemma}
\begin{proof}
For $p$ being idempotent, note that a direct calculation entails
\begin{align*}
p^2 &=   \frac{1}{ \vert \Lambda \vert^2 } \sum_{\gamma, \gamma' \in \Lambda} \overline{ u(\gamma ) u(\gamma' ) } \pi(\gamma) \pi( \gamma')
=   \frac{1}{ \vert \Lambda \vert^2 } \sum_{\gamma, \gamma' \in \Lambda} \overline{ u(\gamma ) u(\gamma' ) } \sigma(\gamma, \gamma') \pi(\gamma  \gamma') \\
&=   \frac{1}{ \vert \Lambda \vert^2 } \sum_{\gamma, \gamma' \in \Lambda} \overline{  u(\gamma \gamma') } \pi(\gamma  \gamma')
=  \frac{1}{ \vert \Lambda \vert  } \sum_{\gamma  \in \Lambda} \overline{  u(\gamma  ) } \pi(\gamma  ) = p.
\end{align*}
For $p$ being self-adjoint, the assumption that $\sigma$ is a coboundary on $\Lambda \times \Lambda$
yields that
\[
\begin{split}
(  \overline{ u(\gamma) }  \pi(\gamma)  )^\ast = &  u(\gamma) \pi(\gamma)^\ast = u(\gamma) \overline{\sigma(\gamma, \gamma^{-1})} \pi(\gamma^{-1} ) \\
 = &   u(\gamma)  \overline{ u(\gamma) u(\gamma^{-1} ) } u(e)  \pi(\gamma^{-1} )   =  \overline{   u(\gamma^{-1} ) }   \pi(\gamma^{-1} ),
\end{split}
\]
for $\gamma \in \Lambda$. Hence,
\[
p^\ast =   \frac{1}{ \vert \Lambda \vert } \sum_{\gamma \in \Lambda } ( \overline{ u(\gamma ) }  \pi(\gamma) ) ^\ast =
 \frac{1}{ \vert \Lambda \vert } \sum_{\gamma \in \Lambda}   \overline{ u(\gamma^{-1})}     \pi(\gamma^{-1} )  = p.
\]
This concludes the proof.
\end{proof}

\begin{proposition} \label{prop:TAction}
 Let $(\pi, \Hpi)$ be a discrete series $\sigma$-representation of $G$ and let $\Gamma \leq G$ be a lattice. Then for every $T \in \pi(\Gamma)''$, there exists $c \in \ell^2 (\Gamma)$ such that
 \begin{align} \label{eq:TAction}
  T \eta = \sum_{\gamma \in \Gamma} c_{\gamma} \pi(\gamma) \eta, \quad \text{for} \quad \eta \in \cB_{\pi}.
 \end{align}
In addition, suppose that $\Lambda \subseteq \Gamma$ is a finite subgroup such that there exists $u: \Lambda \rightarrow \mathbb{C}$ for which $\sigma(\gamma, \gamma') = u(\gamma) u(\gamma') \overline{u(\gamma \gamma')}$ for all $\gamma, \gamma' \in \Lambda$. Let
\[
p = \frac{1}{|\Lambda|} \sum_{\gamma' \in \Lambda}  \overline{u(\gamma')} \pi(\gamma').
\]
Then, if  $c = (c_\gamma)_{\gamma \in \Gamma} \in \ell^2(\Gamma)$ is a sequence satisfying \eqref{eq:TAction} for $T \in \pi(\Gamma)''$, then the sequence $c' = (c_\gamma' )_{\gamma \in \Gamma}$ defined by
\[
c'_\gamma = \frac{1}{|\Lambda|} \sum_{\gamma' \in \Lambda} c_{\gamma \gamma'} \overline{ u((\gamma')^{-1}) } \sigma(\gamma \gamma', (\gamma')^{-1}),
\]
 satisfies $Tp \eta = \sum_{\gamma \in \Gamma}  c_\gamma'  \pi(\gamma) \eta$ for all $\eta \in \cB_\pi$
\end{proposition}

 \begin{proof}
By Lemma \ref{lem:intertwiner}, it may be assumed that $\Hpi$ is a closed subspace of $\oplus_{i \in I} \ell^2(\Gamma)$ and that  $\pi|_{\Gamma}$ equals $\oplus_{i \in I} \lambda_{\Gamma}^{\sigma}$ for some index set $I$.
 Note that the Bessel vectors of $\pi$ remain Bessel vectors under this identification.  Let $p_\pi$ be the orthogonal projection of $\oplus_{i \in I} \ell^2(\Gamma)$  onto $\Hpi$. Then
\begin{align*}
\pi(\Gamma)'' &=   \big(\{ \oplus_{i \in I} \lambda_{\Gamma}^{\sigma} (\gamma) \mid \gamma \in \Gamma \} p_\pi \big) ''
 =
 \{ \oplus_{i \in I} \lambda_{\Gamma}^{\sigma} (\gamma) \mid \gamma \in \Gamma \}'' p_\pi \\
 &=
 \{ \oplus_{i \in I} T \mid T \in \VN \} p_\pi.
\end{align*}
Let  $\eta \in \cB_\pi$   and let $T \in \VN$. By Kaplansky's density theorem, there exists a bounded net $(T_{\alpha})_\alpha$ in $ {\rm span} \: \lambda_{\Gamma}^{\sigma} (\Gamma)$ that converges to $T$ strongly.  It follows, using boundedness of the net in case $I$ is infinite,  that $\oplus_{i \in I} T_\alpha \rightarrow \oplus_{i \in I} T$ strongly.
 The operators $T$ and $T_{\alpha}$ can be written as
  \begin{align} \label{eq:expand}
      T = \sum_{\gamma \in \Gamma} c_{\gamma} \lambda_{\Gamma}^{\sigma} (\gamma), \qquad   T_\alpha = \sum_{\gamma \in \Gamma} c_{\alpha, \gamma} \lambda_{\Gamma}^{\sigma} (\gamma),
  \end{align}
  where each $(c_{\alpha, \gamma})_{\gamma \in \Gamma}$ is finite supported and the first sum converges in the $L^2$-topology (see Section \ref{Sect=GNS}).
 We have
   \[
\lim_{\alpha} \Vert   (c_{\alpha, \gamma})_\gamma - (c_\gamma)_\gamma \Vert_{\ell^2} =  \lim_{\alpha}
  \Vert (T - T_\alpha) \delta_e \Vert_{\ell^2} = 0.
   \]
   Therefore, for $\eta \in \cB_\pi$ and $\eta' \in \oplus_{i \in I} \ell^2(\Gamma)$, using that $ (\langle \pi(\gamma) \eta, \eta' \rangle)_{\gamma \in \Gamma} \in \ell^2(\Gamma)$, it follows that
  \[
      \langle (\oplus_{i \in I }T_\alpha) \eta, \eta' \rangle = \sum_{\gamma \in \Gamma} c_{\alpha, \gamma}  \langle \oplus_{i \in I} \lambda_{\Gamma}^{\sigma} (\gamma) \eta, \eta' \rangle =  \sum_{\gamma \in \Gamma} c_{\alpha, \gamma}  \langle \pi(\gamma) \eta, \eta' \rangle \rightarrow \sum_{\gamma \in \Gamma} c_\gamma \langle \pi(\gamma) \eta, \eta' \rangle.
  \]
  On the other hand, certainly  $\langle (\oplus_{i \in I }T_\alpha) \eta, \eta' \rangle \rightarrow \langle (\oplus_{i \in I } T) \eta, \eta' \rangle$ as $(\oplus_{i \in I }T_\alpha) \rightarrow (\oplus_{i \in I }T)$ strongly.
  This shows that $\sum_{\gamma \in \Gamma} c_{\gamma} \pi(\gamma) \eta = (\oplus_{i \in I } T) \eta$ for every $\eta \in \cB_\pi$.

  For the additional part, suppose that $\sigma(\gamma, \gamma') = u(\gamma) u(\gamma') \overline{u(\gamma \gamma')}$ for all $\gamma, \gamma' \in \Lambda$. Define the projection $q = \vert \Lambda \vert^{-1} \sum_{\gamma' \in \Lambda}  \overline{u(\gamma')} \lambda_{\Gamma}^{\sigma}(\gamma')$.
    Since there exists an equivalence of $\sigma$-representations $\pi \simeq (\oplus_{i \in I} \lambda_{\Gamma}^{\sigma})  p_\pi$, it follows that $p \simeq  (\oplus_{i \in I} q)  p_\pi$.  Therefore, $(\oplus_{i \in I }T) p  \simeq (\oplus_{i \in I } T q) p_\pi$, with
\begin{align*}
    T q  = & \frac{1}{|\Lambda|} \sum_{\gamma \in \Gamma} \sum_{\gamma' \in \Lambda}  c_\gamma \overline{u(\gamma')} \lambda_{\Gamma}^{\sigma} (\gamma)\lambda_{\Gamma}^{\sigma}(\gamma')
    = \frac{1}{|\Lambda|}
     \sum_{\gamma \in \Gamma} \sum_{\gamma' \in \Lambda}  c_\gamma   \overline{u(\gamma')} \sigma(\gamma, \gamma')  \lambda_{\Gamma}^{\sigma} (\gamma \gamma') \\
  = &  \frac{1}{|\Lambda|} \sum_{\gamma \in \Gamma} \sum_{\gamma' \in \Lambda}  c_{\gamma  \gamma'} \overline{u((\gamma')^{-1})} \sigma(\gamma \gamma', (\gamma')^{-1})  \lambda_{\Gamma}^{\sigma} (\gamma).
\end{align*}
The proof can now be completed as follows. The previous equality gives an expansion for $Tq$ of the form \eqref{eq:expand} but with coefficients $c_\gamma$ replaced by $c_\gamma'  = \frac{1}{|\Lambda|}\sum_{\gamma' \in \Lambda}  c_{\gamma  \gamma'} \overline{u((\gamma')^{-1})} \sigma(\gamma \gamma', (\gamma')^{-1})$. Repeating the first part of the proof again for this expansion yields the desired result.
  \end{proof}

The proof of the first part of Proposition \ref{prop:TAction} is an adaption of \cite[Theorem 5.1]{romero2020density}, and provides an extension of that result to possibly non-factor von Neumann algebras. The additional part will play an important role below

\subsection{Hilbert modules} \label{sec:hilbert_lattices}
If $(\pi, \Hpi)$ is a discrete series $\sigma$-representation and $\eta \in \Hpi$ is a unit vector, then the orthogonality relations \eqref{eq:ortho} imply that
the linear mapping
\[
 V_{\eta} : \Hpi \to L^2 (G), \quad V_{\eta} f (x) = d^{1/2}_{\pi} \langle f, \pi(x) \eta \rangle
\]
is an isometry that interwines $\pi$ and the twisted left-regular representation $\lambda_G^{\sigma}$ on $L^2 (G)$, defined by
\[
 \lambda_G^{\sigma} (y) F(x) = \sigma(y, y^{-1} x) F(y^{-1} x), \quad x,y \in G.
\]
In particular, this shows that $\pi$ is unitarily equivalent to a subrepresentation of $\lambda_G^{\sigma}$.

For a lattice $\Gamma \leq G$, the restriction $\lambda_G^{\sigma} |_{\Gamma}$ is unitarily equivalent to the tensor product $\lambda_{\Gamma}^{\sigma} \otimes \id$ acting on $\ell^2 (\Gamma) \otimes L^2 (F)$, where $F \subset G$ is a fundamental domain of $\Gamma \leq G$. Therefore, the
von Neumann algebra $\lambda_G^{\sigma} (\Gamma)'' \subset B(L^2 (G))$ is isomorphic to $\VN$,
and the image space $V_{\eta} (\Hpi) \leq L^2 (G)$ is a submodule of the Hilbert $\VN$-module $L^2 (G) \cong \ell^2 (\Gamma) \otimes L^2 (F)$.
The representation space $\Hpi$ forms thus a Hilbert $\VN$-module with action defined by
\[
 T f = V_{\eta}^* T V_{\eta} f
\]
for $T \in \VN$ and $f \in \Hpi$.

The following result summarizes parts of the above discussion and determines the center-valued von Neumann dimension of $\Hpi$ as a Hilbert $\VN$-module, cf.
\cite[Theorem 1]{bekka2004square} and \cite[Theorem 1.3]{enstad2021density}
for further details.

\begin{theorem}[\cite{bekka2004square, enstad2021density}] \label{thm:cdim_unimodular}
Let $(\pi, \Hpi)$ be a discrete series $\sigma$-representation of $G$ of formal dimension $d_{\pi} > 0$. Suppose $\Gamma \leq G$ is a lattice. Then $\pi|_{\Gamma}$ extends to give $\Hpi$ the structure of a Hilbert $\vN$-module.
The center-valued von Neumann dimension $\cdim_{\vN} (\Hpi)$ of $\Hpi$ is given by
the $\sigma$-twisted convolution operator on $\ell^2 (\Gamma)$,
\[
f \mapsto \phi \ast_{\sigma} f := \sum_{\gamma' \in \Gamma} \phi (\gamma') \lambda_{\Gamma}^{\sigma} (\gamma') f
\]
with kernel
\begin{align*}
\phi (\gamma)
=
\begin{cases}
\frac{d_{\pi}}{|C_{\gamma}|} \int_{G/Z(\gamma)} \overline{\sigma(\gamma, y)} \sigma(y, y^{-1} \gamma y) \langle \eta, \pi(y^{-1} \gamma y) \eta \rangle \; d (y Z(\gamma)),
& \text{if $C_{\gamma}$ is finite $\sigma$-regular;} \\
0, & \text{otherwise},
\end{cases}
\end{align*}
where $\eta \in \Hpi$ is a unit vector and $C_{\gamma}$ (resp. $Z(\gamma))$ denotes the conjugacy class (resp. centralizer) of $\gamma$ in $\Gamma$.
\end{theorem}

The center-valued von Neumann dimension will be computed for semi-simple Lie groups in Section \ref{eq:cocycle_lattice}.

 \subsection{Frames and Riesz sequences}
Let $\eta \in \Hpi \setminus \{0\}$. A family $(\pi(\gamma) \eta )_{\gamma \in \Gamma}$ is a \emph{frame} for $\Hpi$ if
\begin{align} \label{eq:frame_ineq}
A \| f \|_{\Hpi}^2 \leq \sum_{\gamma \in \Gamma} |\langle f, \pi(\gamma) \eta \rangle|^2 \leq B \| f \|_{\Hpi}^2, \quad f \in \Hpi.
\end{align}
for some constants $0 < A \leq B < \infty$.
If $(\pi(\gamma) \eta)_{\gamma \in \Gamma}$ is a frame for $\Hpi$, then $\eta$ is clearly a cyclic vector for $\pi|_{\Gamma}$.
A family satisfying the upper bound in \eqref{eq:frame_ineq} is called a \emph{Bessel sequence} in $\Hpi$.

A family $(\pi(\gamma) \eta )_{\gamma \in \Gamma}$ is a \emph{Riesz sequence} in $\Hpi$ if there exist constants $0 < A \leq B < \infty$ such that
\[
  A \| c \|_{\ell^2}^2 \leq \bigg\| \sum_{\gamma \in \Gamma} c_{\gamma} \pi (\gamma) \eta \bigg\|_{\Hpi}^2 \leq B \| c \|_{\ell^2}^2, \quad c = (c_{\gamma})_{\gamma \in \Gamma} \in \ell^2 (\Gamma).
\]
In particular, a Riesz sequence $(\pi(\gamma) \eta )_{\gamma \in \Gamma}$ is a Bessel sequence and $\mathbb{C}$-linearly independent.
Hence, if the projective $\Gamma$-stabilizer of $\eta$,
\[ \Gamma_{[\eta]} := \big\{ \gamma \in \Gamma : \pi(\gamma) \eta \in \mathbb{C} \eta \big\} \]
is non-trivial, then $(\pi(\gamma) \eta )_{\gamma \in \Gamma}$ cannot be a Riesz sequence. For this reason, also the system $(\pi (\gamma) \eta )_{\gamma \in \Lambda_{\eta}}$, where $\Lambda_{\eta}$ is a set of representatives of $\Gamma / \Gamma_{[\eta]}$ will be considered. It is said to be a \emph{Riesz sequence} if
\[
 A \| c \|_{\ell^2}^2 \leq \bigg\| \sum_{\gamma \in \Lambda_{\eta}} c_{\gamma} \pi (\gamma) \eta \bigg\|_{\Hpi}^2 \leq B \| c \|_{\ell^2}^2, \quad c = (c_{\gamma})_{\gamma \in \Lambda_{\eta}} \in \ell^2 (\Lambda_{\eta}).
\]
Note that if  $(\pi (\gamma) \eta )_{\gamma \in \Lambda_{\eta}}$ is a Riesz sequence for some family $\Lambda_{\eta}$ of coset representatives of $\Gamma / \Gamma_{\eta}$, then it is a Riesz sequence for all such families.

\begin{proposition} \label{prop:riesz_separating}
Let $(\pi, \Hpi)$ be a discrete series $\sigma$-representation of $G$ and let $\Gamma \leq G$ be a lattice. Suppose that $\eta \in \Hpi \setminus \{0\}$ is such that $\Gamma_{[\eta]}$ is finite and let $p_{\eta} := |\Gamma_{[\eta]} |^{-1} \sum_{\gamma \in \Gamma_{[\eta]}} \overline{u(\gamma)} \pi(\gamma)$, where $u : \Gamma_{[\eta]} \to \mathbb{C}$ is such that $\pi(\gamma) \eta = u(\gamma) \eta$. Let $\Lambda_{\eta}$ be coset representatives of $\Gamma / \Gamma_{[\eta]}$.

If $(\pi(\gamma) \eta)_{\gamma \in \Lambda_{\eta}}$ is a Riesz sequence in $\Hpi$, then $\eta$ is a $ p_{\eta}$-separating vector.
\end{proposition}
\begin{proof}
For $\gamma, \gamma' \in \Gamma_{[\eta]}$, it follows that $\pi(\gamma \gamma') \eta = u(\gamma \gamma') \eta$, and
\[
\pi(\gamma \gamma') \eta = \overline{\sigma(\gamma, \gamma')} \pi(\gamma) \pi(\gamma') \eta = \overline{\sigma(\gamma, \gamma')} u(\gamma) u(\gamma') \eta.
\]
Hence, $\sigma|_{\Gamma_{[\eta]} \times \Gamma_{[\eta]}}$ forms a coboundary of the form $\sigma(\gamma, \gamma') = u(\gamma) u(\gamma') \overline{u(\gamma \gamma')}$ for $\gamma, \gamma' \in \Gamma_{[\eta]}$.
 Hence, $p_{\eta}$ is a projection by Lemma \ref{lem:projection}.
The proof will be split into three steps.

\textbf{Step 1.} \emph{(Bessel vector)}.
Let $c = (c_{\gamma})_{\gamma \in \Gamma} \in \ell^2 (\Gamma)$. Since $\pi( \gamma' ) \eta  = u(\gamma') \eta$ for all $\gamma' \in \Gamma_{[\eta]}$, it follows that
\begin{align} \label{eq:periodization}
 \sum_{\gamma \in \Gamma}  c_\gamma \pi(\gamma) \eta = \sum_{\gamma \in \Lambda_{\eta}} \sum_{\gamma' \in \Gamma_{[\eta]}}  c_{\gamma \gamma'} u(\gamma') \overline{\sigma(\gamma, \gamma')} \pi(\gamma) \eta =
 \sum_{\gamma \in \Lambda_{\eta}} d_{\gamma} \pi(\gamma) \eta,
\end{align}
where $d_{\gamma} := \sum_{\gamma' \in \Gamma_{[\eta]}} c_{\gamma \gamma'} u(\gamma') \overline{\sigma(\gamma, \gamma')}$ for $\gamma \in \Lambda_{\eta}$. An application of Cauchy-Schwarz yields that $\| d \|_{\ell^2 (\Lambda_{\eta})} \leq |\Gamma_{[\eta]} |^{\frac{1}{2}} \| c \|_{\ell^2(\Gamma)}$.
Therefore, if $\pi(\Lambda_{\eta}) \eta$ is a Riesz sequence in $\Hpi$ with upper bound $B > 0$,
then, in particular,
 \begin{align*}
   \bigg\| \sum_{\gamma \in \Gamma}  c_\gamma \pi(\gamma) \eta \bigg\|^2_{\Hpi}
   &=
      \bigg\| \sum_{\gamma \in \Lambda_{\eta}} d_{\gamma} \pi(\gamma) \eta \bigg\|^2_{\Hpi} \leq B \| d\|_{\ell^2 (\Lambda_{\eta})}^2 \leq |\Gamma_{[\eta]} |^{\frac{1}{2}} B \| c \|^2_{\ell^2},
   \end{align*}
showing that $\pi(\Gamma) \eta$ is a Bessel sequence, i.e., $\eta \in \mathcal{B}_{\pi}$.

\textbf{Step 2.} \emph{(Shifted Riesz sequence).}
Let $\Lambda_{\eta}$ be any set of representatives of the equivalence classes $\Gamma \slash \Gamma_{[\eta]}$ such that $\pi(\Lambda_{\eta}) \eta$ is a Riesz sequence. Note that if   $\nu \in \Gamma_{[\eta]}$, then also $\Lambda_{\eta} \nu$ is a set of representatives of $\Gamma \slash \Gamma_{[\eta]}$ such that
  \[
  \pi(\Lambda_{\eta} \nu) \eta = (  \overline{ \sigma(\gamma, \nu)} u(\nu)  \pi(\gamma) \eta )_{\gamma \in \Lambda_{\eta}},
  \]
  is a Riesz sequence. This observation allows to apply the argument in the next step to $\Lambda_{\eta}$ replaced by $\Lambda_{\eta} \nu$ as well.

  \textbf{Step 3.} \emph{(Separating vector).}
For showing that $\eta$ is $ p_{\eta}$-separating, let $T \in \pi(\Gamma)''$ be such that $T p_{\eta} = T$ and  $T \eta = 0$.
It will be shown that $T p_{\eta}  = 0$. Choose $c, c'  \in \ell^2 (\Gamma)$ as provided by Proposition \ref{prop:TAction}.
 Then the identity \eqref{eq:periodization} gives
  \begin{align*}
   0 &= T \eta =  \sum_{\gamma \in \Lambda_{\eta}} \sum_{\gamma' \in \Gamma_{[\eta]}}  c_{\gamma \gamma'} u(\gamma') \overline{\sigma(\gamma, \gamma')} \pi(\gamma) \eta =
 \sum_{\gamma \in \Lambda_{\eta}} d_{\gamma} \pi(\gamma) \eta.
  \end{align*}
 By assumption, $\pi(\Lambda_{\eta}) \eta$ is a Riesz sequence in $\Hpi$. Therefore, if $A > 0$ is a lower Riesz bound, then
 \[
  0 = \bigg\| \sum_{\gamma \in \Lambda_{\eta}} d_{\gamma} \pi (\gamma) \eta \bigg\|_{\Hpi}^2 \geq A \| d \|^2_{\ell^2 (\Lambda_{\eta})},
 \]
and thus it follows that, for all $\gamma \in \Lambda_{\eta}$,
  \begin{equation}\label{Eqn=CoefficientZero}
    d_{\gamma} =  \sum_{\gamma' \in \Gamma_{[\eta]} }   c_{\gamma \gamma'}  \overline{\sigma(\gamma, \gamma')}   u(\gamma') = 0.
  \end{equation}
  Possibly replacing $\Gamma_0$ by $\Gamma_0 \nu$ for some $\nu \in \Gamma_{[\eta]}$ (cf. Step 2) yields that $d_{\gamma} = 0$ for all $\gamma \in \Gamma$.

Using the sequence $c'$ of Proposition \ref{prop:TAction}, it also follows that
  \begin{equation}\label{Eqn=Coefficients2}
  \begin{split}
    T p_{\eta} \eta = &  \sum_{\gamma \in \Gamma}  c_\gamma'  \pi(\gamma) \eta  =
     \sum_{\gamma \in \Gamma} \frac{1}{|\Gamma_{[\eta]}|} \sum_{\gamma' \in \Gamma_{[\eta]}}  c_{\gamma  \gamma'} \overline{u((\gamma')^{-1})} \sigma(\gamma \gamma', (\gamma')^{-1})  \pi(\gamma) \eta.
  \end{split}
  \end{equation}
  Since $\sigma (\gamma', \nu) = u(\gamma') u(\nu) \overline{ u(\gamma \nu) }   $ for $\gamma, \nu \in \Gamma_{[\eta]}$, a direct calculation gives
  \[
  \sigma(\gamma', (\gamma')^{-1}) = u(\gamma') u((\gamma')^{-1}) \overline{u( e_G )} =  u(\gamma') u((\gamma')^{-1}).
  \]
  This, together with the cocycle identity \eqref{eq:cocycle_identity} yields that, for all $\gamma \in \Gamma, \gamma' \in \Gamma_{[\eta]}$,
\[
\sigma(\gamma \gamma', (\gamma')^{-1}) = \overline{\sigma(\gamma, \gamma')} \sigma(\gamma', (\gamma')^{-1}) \sigma(\gamma, e_G) =
 \overline{\sigma(\gamma, \gamma')} u(\gamma') u((\gamma')^{-1}).
\]
Hence, for the coefficients $c'$ in \eqref{Eqn=Coefficients2}, it follows by the previous equality and the identity \eqref{Eqn=CoefficientZero} that
\[
 \sum_{\gamma' \in \Gamma_{[\eta]}} c_{\gamma  \gamma'} \overline{u((\gamma')^{-1})} \sigma(\gamma \gamma', (\gamma')^{-1}) =  \sum_{\gamma' \in \Gamma_{[\eta]}} c_{\gamma  \gamma'} u(\gamma') \overline{\sigma(\gamma, \gamma')} = d_{\gamma} = 0
\]
for all $\gamma \in \Gamma$.  Thus, $T p_{\eta} = 0$.
  \end{proof}

Proposition \ref{prop:riesz_separating} provides an extension of \cite[Proposition 5.2]{romero2020density} to possibly non-trivial projective stabilizers.

\section{Discrete series representations of semi-simple Lie groups} \label{sec:discrete_semi}
Throughout, $G$ denotes a connected non-compact semi-simple Lie group with finite center and no compact factors,
i.e., no non-trivial connected normal compact subgroups.

\subsection{Center-valued trace and von Neumann dimension} \label{eq:cocycle_lattice}
In this section the center-valued trace and von Neumann dimension
will be determined for twisted group von Neumann algebras of lattices in semi-simple Lie groups.
Both results rely on two lemmata on cocycles and lattices that are specific
for the setting of semi-simple Lie groups.

The first lemma shows that any central element is automatically $\sigma$-regular.

\begin{lemma} \label{lem:central_sregular}
Let $\sigma$ be a cocycle on $G$. If $x \in Z(G)$, then $\sigma(x,y) = \sigma (y,x)$ for all $y \in G$.
\end{lemma}
\begin{proof}
Let $\widetilde{G}$ be the universal covering group of $G$ with covering map $q : \widetilde{G} \to G$ . Let $s : G \to \widetilde{G}$ be a Borel cross-section with $s(e_G) = e_{\widetilde{G}}$, so that $q \circ s = \id_{G}$.
By \cite[Proposition 3.4]{moore1963extensions} or \cite[Corollary 3.1]{bagchi2003homogeneous}, there exists a unique character $\chi \in \widehat{\ker(q)}$ such that
\[
\sigma(x,y) = \chi (s(x) s(y) s(xy)^{-1}), \quad x,y \in G.
\]
Since $Z(\widetilde{G}) = q^{-1} (Z(G))$, see, e.g., \cite[Proposition 9.5.2]{hilgert2012structure}, it follows that if $x \in Z(G)$, then $s(x) \in Z(\widetilde{G})$. Hence, if $x \in Z(G)$ and $y \in G$, then
\[
\sigma(x,y) \overline{\sigma(y,x)} = \chi(s(x) s(y) s(xy)^{-1} s(yx) s(x)^{-1} s(y)^{-1} ) = \chi(e_{\widetilde{G}}) = 1,
\]
as asserted.
\end{proof}

A direct consequence of Lemma \ref{lem:central_sregular} is that if $(\pi, \Hpi)$ is an irreducible $\sigma$-representation of $G$ and $x \in Z(G)$, then
\[
\pi(x) \pi(y) = \sigma(x,y) \pi(yx) = \sigma(x,y) \overline{\sigma(y,x)} \pi(y) \pi (x) = \pi(y) \pi(x), \quad y \in G.
\]
Therefore, Schur's lemma yields that $\pi(x) \in \mathbb{C} \id_{\Hpi}$ for $x \in Z(G)$, so that the restriction $\pi|_Z$ forms a $1$-dimensional $\sigma$-representation $u : Z(G) \to \mathbb{T}$.
In particular, the restriction of $\sigma$ to $Z(G) \times Z(G)$ is a coboundary, i.e.,
\begin{align} \label{eq:central_character}
 \sigma(x,y) = u(x) u(y) \overline{u(xy)}, \quad x,y \in Z(G),
 \end{align}
for some measurable function $u : Z(G) \to \mathbb{T}$.
See also \cite[Lemma 7.2]{kleppner1965multipliers} for an alternative argument for possibly reducible representations.

The following lemma is \cite[Lemma 3.3.1]{goodman1989coxeter}; see also \cite[Theorem 2]{bekka2004square} for a more general version.

\begin{lemma}[\cite{goodman1989coxeter}] \label{lem:finite_central}
Let $\Gamma \leq G$ be a lattice. If $\gamma \in \Gamma$ has finite conjugacy class, then $\gamma \in Z(G)$.
\end{lemma}

In particular, Lemma \ref{lem:finite_central} yields that $Z(\Gamma) = Z(G) \cap \Gamma$.

\begin{lemma} \label{lem:Ctr}
Let $\Gamma \leq G$ be a lattice and let $\sigma$ be a cocycle on $\Gamma$.
Then
\begin{equation}\label{Eqn=Center}
\mathcal{Z} (\VN) = \VN \cap \VN'   = \mathrm{VN}_{\sigma}(  Z (\Gamma )  ).
\end{equation}
Therefore, the center-valued trace is determined by
\begin{equation}
\ctr \bigg(  \sum_{\gamma \in \Gamma}  c_\gamma \lambda_{\Gamma}^{\sigma}(\gamma ) \bigg)= \sum_{\gamma \in Z (\Gamma)}  c_\gamma \lambda_{\Gamma}^{\sigma}(\gamma ),
\end{equation}
for any scalars $c_\gamma \in \mathbb{C}$ that are non-zero for finitely many $\gamma \in \Gamma$.
\end{lemma}
\begin{proof}
Firstly, by Lemma \ref{lem:finite_central}, if the conjugacy class of an element $\gamma \in \Gamma$ is finite, then $\gamma \in Z(\Gamma) = Z(G) \cap \Gamma$ and its conjugacy class in $\Gamma$ contains $\gamma$ only.
In particular, this implies that $\gamma \in Z(\Gamma)$ is $\sigma$-regular by Lemma \ref{lem:central_sregular}.
Therefore, by \cite[Theorem 1 and Lemma 2]{kleppner1962structure} or \cite[Lemma 2.2 and 2.4]{omland2014primeness}, the center $\mathcal{Z} (\VN)$ is spanned linearly by all $\lambda_{\Gamma}^{\sigma} (\gamma)$ with $\gamma \in Z(\Gamma)$. This proves \eqref{Eqn=Center}.
\end{proof}

\begin{proposition} \label{prop:cdim_semisimple}
Let $(\pi, \Hpi)$ be a discrete series $\sigma$-representation of $G$.
Let $\Gamma \leq G$ be a lattice and let $u : Z(\Gamma) \to \mathbb{T}$ be any function such that $\pi(\gamma) = u(\gamma) \id_{\Hpi}$ for $\gamma \in Z(\Gamma)$.
Then the operator $\cdim_{\vN}(\Hpi)$ on $\ell^2 (\Gamma)$ is given by twisted convolution
with kernel
\[
\phi (\gamma) =
\begin{cases}
 d_{\pi} \vol(G/\Gamma) \overline{u (\gamma)}, & \text{if} \;\; \gamma \in Z(\Gamma), \\
 0, & \text{otherwise.}
\end{cases}
\]
\end{proposition}
\begin{proof}
If $\gamma \in \Gamma$ is such that $C_{\gamma}$ is finite, then $\gamma \in Z (\Gamma) = \Gamma \cap Z(G)$ by Lemma \ref{lem:central_sregular}. In particular, $|C_{\gamma} | = 1$ and $Z(\gamma) = \Gamma$, so that
the function $\phi$ of Theorem \ref{thm:cdim_unimodular} can be computed as
\begin{align*}
\phi(\gamma) &= d_{\pi} \int_{G/\Gamma} \overline{\sigma(\gamma, y)} \sigma(y, \gamma) \langle \eta, \pi (\gamma) \eta \rangle \; d\mu_{G/\Gamma}(y\Gamma)
= d_{\pi} \overline{  u(\gamma)} \int_{G/\Gamma} \overline{\sigma(\gamma, y)} \sigma(y,\gamma) d\mu_{G/\Gamma} (y\Gamma) \\
&= d_{\pi} \vol(G / \Gamma) \overline{ u(\gamma)},
\end{align*}
where the last equality follows since $\overline{\sigma(\gamma, y)} \sigma(y,\gamma) = 1$ for all $y \in G$ by Lemma \ref{lem:central_sregular}.
\end{proof}

 Proposition \ref{prop:cdim_semisimple} provides a generalization of \cite[Theorem 2(ii)]{bekka2004square} to projective representations
 in the case of a semi-simple real Lie group.

\subsection{Density conditions} The following two results are the main results of this paper.

\begin{theorem} \label{thm:density_main1}
Let $(\pi, \Hpi)$ be a discrete series $\sigma$-representation of $G$ and let $\eta \in \Hpi \setminus \{0\}$.  Let $\Gamma \leq G$ be a lattice.
Suppose that $\Lambda \subset \Gamma$ is a finite subgroup containing $Z(\Gamma)$ and such that there exists $u : \Lambda \to \mathbb{T}$ satisfying
$\pi(\gamma) \eta = u(\gamma) \eta$ for all $\gamma \in \Lambda$. Let $p_{\Lambda} := |\Lambda|^{-1} \sum_{\gamma \in \Lambda} \overline{u(\gamma)} \pi(\gamma)$.

Then the following assertions hold:
\begin{enumerate}
\item[(i)] If $\eta$ is a cyclic vector for $\pi|_{\Gamma}$, then $\vol(G/\Gamma) d_{\pi} \leq |\Lambda|^{-1}$.
\item[(ii)] If $\eta$ is a $p_{\Lambda}$-separating vector for $\pi(\Gamma)''$, then $\vol(G/\Gamma) \geq |\Lambda|^{-1}$.
\end{enumerate}
\end{theorem}
\begin{proof}
As in the proof of Proposition \ref{prop:riesz_separating}, it can be shown that $\sigma|_{\Lambda \times \Lambda}$ is a coboundary of the form $\sigma(\gamma, \gamma') = u(\gamma) u(\gamma') \overline{u(\gamma \gamma')}$ for $\gamma, \gamma' \in \Lambda$.
An application of Lemma \ref{lem:projection} yields therefore that
\[
p :=
p_{\Lambda} = \frac{1}{ \vert \Lambda \vert } \sum_{\gamma \in \Lambda} \overline{ u(\gamma ) }  \pi (\gamma )
\]
is a projection in $\pi (\Gamma)''$.

Note that by  Section \ref{sec:hilbert_lattices} or Lemma \ref{lem:intertwiner}, the map  $\pi$ determines a normal $\ast$-homomorphism $\lambda_\Gamma^{\sigma}(\gamma) \mapsto \pi(\gamma)$ from $\vN$ into $\pi (\Gamma)''$, which will still be denoted by $\pi$. The kernel of $\pi$ is a $\sigma$-weakly closed two-sided ideal in $\VN$, and hence there exists a central projection $q \in \vN$ such that $\ker(\pi) = (1- q ) \VN$.  Thus, $\vN q \simeq \pi (\Gamma)''$ through $\pi$.
This way $\Hpi$  is a faithful $\pi (\Gamma)''$-module as well as a $\VN$-module through $\pi$.
Set the projection in $\vN$ corresponding to $p$  by $p_0  = \frac{1}{ \vert \Lambda \vert } \sum_{\gamma \in \Lambda} \overline{ u(\gamma ) }  \lambda_{\Gamma} ^{\sigma} (\gamma )$, so that $\pi(p_0) = p$.
Since $Z(\Gamma) \subset \Lambda$ by assumption, an application of Lemma \ref{lem:Ctr} gives
\begin{equation}\label{Eqn=CtrComputation}
\ctr(p_0) = \frac{1}{ \vert \Lambda \vert }  \sum_{\gamma \in Z (\Gamma) } \overline{ u(\gamma ) }   \lambda_{\Gamma}^{\sigma} (\gamma) .
\end{equation}
By Proposition \ref{prop:cdim_semisimple},
\begin{equation}\label{Eqn=DimensionComputation}
\cdim_{  \vN   }(  \Hpi  )  =  \vol(G/ \Gamma ) d_{\pi} \sum_{\gamma \in Z (\Gamma) } \overline{  u (\gamma)  } \lambda_{\Gamma}^{\sigma}(\gamma).
\end{equation}
Assertions (i) and (ii) will be shown by combining  \eqref{Eqn=CtrComputation} and  \eqref{Eqn=DimensionComputation}.

(i) Suppose that $\vol(G/\Gamma) d_{\pi} > |\Lambda|^{-1}$. Note that $\pi(p_0) \eta = p  \eta = \eta$, and thus Proposition \ref{Prop=DimEstimate} implies
\begin{equation}\label{Eqn=CtrEstimate}
\cdim_{  \vN   } \big(   \overline{\vN   \eta}    \big) \leq \ctr(p_0).
\end{equation}
This, combined with the identities \eqref{Eqn=CtrComputation} and \eqref{Eqn=DimensionComputation}, yields that
\[
\begin{split}
 \cdim_{  \vN  }\big(   \overline{\vN    \eta}  \big) &\leq
 \frac{1}{ \vert \Lambda \vert} \sum_{\gamma \in Z(\Gamma)} \overline{ u(\gamma) }  \lambda_{\Gamma}^{\sigma} (\gamma) \\
&<     \vol(G / \Gamma ) d_{\pi} \sum_{\gamma \in Z (\Gamma)}  \overline{ u(\gamma) } \lambda_{\Gamma}^{\sigma} (\gamma) \\
 &=    \cdim_{  \vN   }(   \Hpi   ).
\end{split}
\]
In particular, this implies that $\Span \pi (\Gamma)  \eta \neq \Hpi$. By contraposition, this shows (i).

(ii) Suppose that $\eta$ is $p$-separating for $\pi (\Gamma)''$. Then it is $p_0 q$-separating if $\Hpi$ is viewed as a $\vN$-module through $\pi : \vN \to \pi (\Gamma)''$.
 Then Proposition \ref{prop:dimension_pseparating} and the fact that $q$ is central implies that
\[
 \cdim_{\vN} (\Hpi) \geq \ctr(q p_0  ) = q \ctr(p_0 ) .
\]
Combining this with the identities \eqref{Eqn=CtrComputation} and \eqref{Eqn=DimensionComputation} entails
\[
\vol(G/\Gamma) d_{\pi} \sum_{\gamma \in Z (\Gamma) } \overline{  u(\gamma)  } \lambda_{\Gamma}^{\sigma} (\gamma)
\geq q |\Lambda|^{-1} \sum_{\gamma \in Z (\Gamma) } \overline{  u(\gamma)  } \lambda_{\Gamma}^{\sigma} (\gamma),
\]
which is only possible if $\vol(G/\Gamma) d_{\pi} \geq |\Lambda|^{-1}$.
\end{proof}

\begin{theorem} \label{thm:main1_frame}
Under the same assumptions and notation as Theorem  \ref{thm:density_main1}, let $X$ be a set of representatives of $\Gamma / \Lambda$.
\begin{enumerate}
 \item[(i)] If $(\pi (\gamma) \eta)_{\gamma \in \Gamma}$ is a frame for $\Hpi$, then
 $
  \vol(G / \Gamma) d_{\pi_{\alpha}} \leq |\Lambda|^{-1}.
 $
\item[(ii)] If  $(\pi (\gamma) \eta)_{\gamma \in X}$ is a Riesz sequence in $\Hpi$, then $
  \vol(G / \Gamma) d_{\pi_{\alpha}} \geq |\Lambda|^{-1}.
 $
\end{enumerate}
\end{theorem}
\begin{proof}
 (i) If $(\pi (\gamma) \eta )_{\gamma \in \Gamma}$ is a frame for $\Hpi$, then it is clearly complete in $\Hpi$. Hence, the conclusion follows from Theorem \ref{thm:density_main1}.

 (ii) If $(\pi (\gamma) \eta)_{\gamma \in X}$ is a Riesz sequence in $\Hpi$, then necessarily
 $ \Lambda = \Gamma_{[\eta]} $ by linear independence. Hence, Proposition \ref{prop:riesz_separating} implies that
  $\eta$ is $p_{\eta}$-separating for $\pi (\Gamma)''$.
   The conclusion follows therefore from Theorem \ref{thm:density_main1}.
\end{proof}

\section{Holomorphic discrete series and Bergman kernels} \label{sec:holomorphic_discrete}
This section concerns the so-called (scalar-valued) holomorphic discrete series.
Henceforth, it is additionally assumed that $G$ is such that the symmetric space $ G/K$,
where $K \leq G$ is a maximal compact subgroup, is Hermitian, i.e., has an invariant complex structure. Lie groups satisfying these assumptions are sometimes said to be of \emph{Hermitian type}.

\subsection{Holomorphic discrete series}

A Hermitian symmetric space $ G / K$ can be identified (via a holomorphic diffeomorphism) with a bounded symmetric domain $\Omega \subset \mathbb{C}^d$, i.e., a bounded domain such that
 each $w \in \Omega$ is an isolated fixed point of an involutive holomorphic diffeomorphism. Conversely, any bounded symmetric domain $\Omega \subset \mathbb{C}^d$ can be identified with a Hermitian symmetric space of the form $ G/K$, where $G$ is the identity of the group of holomorphic automorphism of $\Omega$ and $K$ an isotropy group. See, e.g., the books \cite{satake1980algebraic, helgason1978differential, hua1979harmonic} for more details and examples.

Let $\Omega \cong G/K$ be a bounded symmetric domain with normalized Euclidean measure $\mu$. The standard Bergman space $A^2 (\Omega)$ on $\Omega$ is defined as
\[
A^2 := A^2(\Omega) = \bigg\{ f \in \mathcal{O} (\Omega) : \| f \|^2_{A^2} := \int_{\Omega} |f(z)|^2 \; d\mu(z) < \infty \bigg\},
\]
where $\mathcal{O} (\Omega)$ denotes the space of holomorphic functions on $\Omega$.
The space $A^2 (\Omega)$ is a reproducing kernel Hilbert space, i.e., the point evaluation $f \mapsto f(z)$ is continuous. Let $k : \Omega \times \Omega \to \mathbb{C}$ be the reproducing kernel in $A^2$, so that
\[
f(z) = \int_{\Omega} f(w) k(z,w) \; d\mu(w), \quad f \in A^2 (\Omega), \; z \in \Omega.
\]
Denote by $d'$ the genus of $\Omega$, that is, $d' := (n + n_1) /r$ with $n$ and $n_1$ being the (complex) dimensions of $\Omega$ resp. of the maximal symmetric domain of tube-type in $\Omega$ and with $r$ being the rank of $\Omega$.
Let
\[
h(z,w) = k (z,w)^{-1/d'}, \quad z, w \in \Omega.
\]
Then $h(z,z)^{\alpha - d'}$ is integrable if, and only if, $\alpha > d' - 1$.

For $\alpha > d'- 1$, let $c_{\alpha}^{-1} := \int_{\Omega} h(z,z)^{\alpha - d'} d\mu(z) < \infty$ and $d\mu_{\alpha} (z) = c_{\alpha} h(z,z)^{\alpha - d'} d\mu(z)$. The associated weighted Bergman space is defined as
\[
 A_{\alpha}^2 := A_{\alpha}^2(\Omega) = \bigg\{ f \in \mathcal{O} (\Omega) : \| f \|^2_{A_{\alpha}^2} := \int_{\Omega} |f(z)|^2 \; d\mu_{\alpha} (z) < \infty \bigg\}.
\]
The space $A^2_{\alpha} (\Omega)$ is a reproducing kernel Hilbert space with reproducing kernel $k^{(\alpha)} (z,w) = k(z,w)^{\alpha /d'}$ satisfying
\begin{align} \label{eq:repro}
f(z) = \int_{\Omega} f(w) k^{(\alpha)} (z,w) \; d\mu_{\alpha} (w) = \langle f, k^{(\alpha)}_z \rangle,
\end{align}
where $k^{(\alpha)}_z (w) :=  \overline{k^{(\alpha)} (z,w)}$. See \cite{koranyi2000function,faraut1990function, stoll1977mean, faraut1994analysis} for more on Bergman kernels on  domains.

Denote by $J_g (w)$ the Jacobian of the action $\Omega \ni w \mapsto g \cdot w \in \Omega$. For  $\alpha > d' - 1$, define $J_g( w)^{\alpha / d'}$ for some choice of the branch of the power if $\alpha \notin \mathbb{Z}$. Then the action of $G$ on a function $f \in A^2_{\alpha} (\Omega)$ given by
\begin{align} \label{eq:projective_holomorphic}
\pi_{\alpha} (g) f (w) = J_{g^{-1}} ( w)^{\alpha / d'} f (g^{-1} \cdot w), \quad g \in G, w \in \Omega,
\end{align}
defines a projective unitary representation $\pi_{\alpha}$ of $G$ on $ A^2_{\alpha} (\Omega)$. It is called the \emph{holomorphic discrete series}.
For each $\alpha > d' -1$, the discrete series $\pi_{\alpha}$ is irreducible and square-integrable.

For details on Bergman spaces and holomorphic discrete series, see
 \cite{knapp1972bounded, koranyi2000function, neeb1999holomorphy, knapp1986representation}.

\subsection{Density conditions for Bergman kernels} \label{sec:main}
In this section we provide proofs of Theorem \ref{thm:main1_intro} and Theorem \ref{thm:main2_intro} stated in the introduction. Both results form a special case (trivial center $Z(G)$) of the following two theorems, which will be obtained from  Theorem \ref{thm:density_main1} and Theorem \ref{thm:main1_frame}.

\begin{theorem} \label{thm:main_holomorphic}
Let $\pi_{\alpha}$ be a holomorphic discrete series of $G$ on $A^2_{\alpha} (\Omega)$ for $\alpha > d'-1$.
Let $\Gamma \leq G$ be a lattice with stabilizer $\Gamma_z$ for $z \in \Omega$ and let $p_z := |\Gamma_z|^{-1} \sum_{\gamma \in \Gamma_z} \overline{u(\gamma)} \pi_{\alpha} (\gamma)$.

Then the following assertions hold:
\begin{enumerate}
 \item[(i)] If $k^{(\alpha)}_{z}$ is a cyclic vector for $\pi_{\alpha}|_{\Gamma}$, then
 $
  \vol(G / \Gamma) d_{\pi_{\alpha}} \leq |\Gamma_z|^{-1}.
 $
\item[(ii)] If $k^{(\alpha)}_{z}$ is a $p_{z}$-separating vector for $\pi_{\alpha} (\Gamma)''$, then $
  \vol(G / \Gamma) d_{\pi_{\alpha}} \geq |\Gamma_z|^{-1}.
 $
\end{enumerate}
\end{theorem}
\begin{proof}
The assertions will be shown by applying Theorem \ref{thm:density_main1} to $\eta = k_{z}^{(\alpha)}$ and $\Lambda = \Gamma_z$.

Since $Z(G)$ is finite and $K$ is a maximal compact subgroup of $G$,
it follows that $Z(G) \subseteq K$.   In addition, since $\Omega \cong G/K$, the stabilizer $\Gamma_z = x_0 K x_0^{-1} \cap \Gamma$ is finite as $ x_0 K x_0^{-1}$ is compact and $\Gamma$ is discrete,
where $z = x_0 K$. In addition, note that Lemma \ref{lem:finite_central} yields that $Z(\Gamma) = Z(G) \cap \Gamma$, and thus $Z(\Gamma) \subseteq Z(G) \subseteq K$, so that $Z(\Gamma) \subseteq \Gamma_z$.

 It remains to show that $\pi_{\alpha} (\gamma) k_z^{(\alpha)} = u(\gamma) k_z^{(\alpha)}$ for some $u : \Gamma_z \to \mathbb{T}$. For this, let $f \in A_\alpha^2(\Omega)$ and $\gamma \in \Gamma$. Then a direct calculation entails
\begin{align*}
 \langle f, \pi_\alpha(\gamma)   k^{(\alpha)}_{z} \rangle
 &=  \overline{ \sigma_\alpha(\gamma, \gamma^{-1}) }  \langle  \pi_\alpha( \gamma^{-1})  f,   k^{(\alpha)}_{z}   \rangle
=      \overline{ \sigma_\alpha(\gamma, \gamma^{-1}) }  ( \pi_\alpha(    \gamma^{-1}   )   f)  (z)  \\
&=     \overline{ \sigma_\alpha(\gamma, \gamma^{-1}) }   J_{\gamma}( z)^{\alpha / d'}  f ( \gamma \cdot z)
=    \overline{ \sigma_\alpha(\gamma, \gamma^{-1}) }   J_{\gamma}( z)^{\alpha / d'} \langle f,   k^{(\alpha)}_{\gamma \cdot z}   \rangle. \numberthis \label{Eqn=Fixed}
\end{align*}
Hence,
$
\pi_\alpha(\gamma)   k^{(\alpha)}_{z} = u(\gamma) k^{(\alpha)}_{z}$ with $u(\gamma) =   \sigma_\alpha(\gamma, \gamma^{-1})  \overline{  J_{\gamma}( z)^{\alpha / d'}}$  for $ \gamma \in \Gamma_z$.
\end{proof}

\begin{theorem}
Under the same assumptions and notation as Theorem \ref{thm:main_holomorphic}, let $\Lambda_z$ be a set of representatives of $\Gamma / \Gamma_z$.
\begin{enumerate}
 \item[(i)] If $(\pi_{\alpha} (\gamma) k^{(\alpha)}_{z})_{\gamma \in \Gamma}$ is a frame for $A^2_{\alpha} (\Omega)$, then
 $
  \vol(G / \Gamma) d_{\pi_{\alpha}} \leq |\Gamma_z|^{-1}.
 $
\item[(ii)] If  $(\pi_{\alpha} (\gamma) k^{(\alpha)}_{z})_{\gamma \in \Lambda_z}$ is a Riesz sequence in $A^2_{\alpha} (\Omega)$, then $
  \vol(G / \Gamma) d_{\pi_{\alpha}} \geq |\Gamma_z|^{-1}.
 $
\end{enumerate}
\end{theorem}
\begin{proof}
As shown in the proof of Theorem \ref{thm:main_holomorphic}, the vector $\eta = k_z^{(\alpha)}$ and subgroup $\Lambda = \Gamma_z$ satisfy the hypotheses of Theorem \ref{thm:density_main1}. Hence, assertions (i) and (ii) follow from Theorem \ref{thm:main1_frame}.
\end{proof}

\section*{Acknowledgements}
 M.C. is supported by the NWO Vidi grant `Non-commutative harmonic analysis and rigidity of operator algebras', VI.Vidi.192.018. J.v.V. gratefully acknowledges support from the Austrian Science Fund (FWF) project J-4445.


\begin{thebibliography}{10}

\bibitem{atiyah1977geometric}
M.~{Atiyah} and W.~{Schmid}.
\newblock {A geometric construction of the discrete series for semisimple Lie
  groups}.
\newblock {\em {Invent. Math.}}, 42:1--62, 1977.

\bibitem{atiyah1976elliptic}
M.~F. {Atiyah}.
\newblock {Elliptic operators, discrete groups and von Neumann algebras}.
\newblock In {\em Colloque ``Analyse et Topologie'' en l'honneur de Henri
  Cartan.}, pages 43--72. Paris: Soci\'et\'e Math\'ematique de France (SMF),
  1976.

\bibitem{bagchi2003homogeneous}
B.~{Bagchi} and G.~{Misra}.
\newblock {The homogeneous shifts.}
\newblock {\em {J. Funct. Anal.}}, 204(2):293--319, 2003.

\bibitem{bekka2004square}
B.~{Bekka}.
\newblock {Square integrable representations, von Neumann algebras and an
  application to Gabor analysis}.
\newblock {\em {J. Fourier Anal. Appl.}}, 10(4):325--349, 2004.

\bibitem{dixmier1981von}
J.~{Dixmier}.
\newblock {\em {Von Neumann algebras.}}, volume~27.
\newblock Elsevier (North-Holland), Amsterdam, 1981.

\bibitem{enstad2021density}
U.~Enstad.
\newblock The density theorem for projective representations via twisted group
  von {Neumann} algebras.
\newblock {\em J. Math. Anal. Appl.}, 511(2):25, 2022.
\newblock Id/No 126072.

\bibitem{faraut1990function}
J.~{Faraut} and A.~{Koranyi}.
\newblock {Function spaces and reproducing kernels on bounded symmetric
  domains}.
\newblock {\em {J. Funct. Anal.}}, 88(1):64--89, 1990.

\bibitem{faraut1994analysis}
J.~{Faraut} and A.~{Kor\'anyi}.
\newblock {\em {Analysis on symmetric cones}}.
\newblock Oxford: Clarendon Press, 1994.

\bibitem{goodman1989coxeter}
F.~M. {Goodman}, P.~{De la Harpe}, and V.~F.~R. {Jones}.
\newblock {\em {Coxeter graphs and towers of algebras}}, volume~14.
\newblock New York etc.: Springer-Verlag, 1989.

\bibitem{helgason1978differential}
S.~{Helgason}.
\newblock {Differential geometry, Lie groups, and symmetric spaces}.
\newblock { Academic Press.}, 1978.

\bibitem{hilgert2012structure}
J.~{Hilgert} and K.-H. {Neeb}.
\newblock {\em {Structure and geometry of Lie groups}}.
\newblock Berlin: Springer, 2012.

\bibitem{hua1979harmonic}
L.~K. {Hua}.
\newblock {\em {Harmonic analysis of functions of several complex variables in
  the classical domains.}}, volume~6.
\newblock American Mathematical Society (AMS), Providence, RI, 1979.

\bibitem{jevtic1996interpolating}
M.~Jevti{\'c}, X.~Massaneda, and P.~J. Thomas.
\newblock Interpolating sequences for weighted {Bergman} spaces of the ball.
\newblock {\em Mich. Math. J.}, 43(3):495--517, 1996.

\bibitem{jones2020bergman}
V.~Jones.
\newblock Bergman space zero sets, modular forms, von neumann algebras and
  ordered groups.
\newblock {\em Preprint. arXiv 2006.16419}.

\bibitem{jones1997subfactors}
V.~{Jones} and V.~S. {Sunder}.
\newblock {\em {Introduction to subfactors}}, volume 234.
\newblock Cambridge: Cambridge University Press, 1997.

\bibitem{kellylyth1999uniform}
D.~{Kelly-Lyth}.
\newblock {Uniform lattice point estimates for co-finite Fuchsian groups.}
\newblock {\em {Proc. Lond. Math. Soc. (3)}}, 78(1):29--51, 1999.

\bibitem{kleppner1962structure}
A.~{Kleppner}.
\newblock {The structure of some induced representations}.
\newblock {\em {Duke Math. J.}}, 29:555--572, 1962.

\bibitem{kleppner1965multipliers}
A.~{Kleppner}.
\newblock {Multipliers on Abelian groups}.
\newblock {\em {Math. Ann.}}, 158:11--34, 1965.

\bibitem{knapp1972bounded}
A.~W. {Knapp}.
\newblock {Bounded symmetric domains and holomorphic discrete series}.
\newblock {Symmetric Spaces, short Courses presented at Washington Univ., pure
  appl. Math. 8, 211-246 (1972).}, 1972.

\bibitem{knapp1986representation}
A.~W. {Knapp}.
\newblock {\em {Representation theory of semisimple groups. An overview based
  on examples}}, volume~36.
\newblock Princeton University Press, Princeton, NJ, 1986.

\bibitem{koranyi2000function}
A.~{Kor\'anyi}.
\newblock {Function spaces on bounded symmetric domains}.
\newblock In {\em Analysis and geometry on complex homogeneous domains}, pages
  183--281. Boston, MA: Birkh\"auser, 2000.

\bibitem{kuhn1992restrictions}
G.~{Kuhn} and T.~{Steger}.
\newblock {Restrictions of the special representation of
  \(\text{Aut}(\text{tree}_ 3)\) to two cocompact subgroups}.
\newblock {\em {Rocky Mt. J. Math.}}, 22(4):1349--1363, 1992.

\bibitem{lisiecki1990kaehler}
W.~{Lisiecki}.
\newblock {Kaehler coherent state orbits for representations of semisimple Lie
  groups}.
\newblock {\em {Ann. Inst. Henri Poincar\'e, Phys. Th\'eor.}}, 53(2):245--258,
  1990.

\bibitem{marco2003density}
N.~Marco and X.~Massaneda.
\newblock On density conditions for interpolation in the ball.
\newblock {\em Can. Math. Bull.}, 46(4):559--574, 2003.

\bibitem{Massaneda2000interpolating}
X.~Massaneda and P.~J. Thomas.
\newblock Interpolating sequences for {Bargmann}-{Fock} spaces in {{\(\mathbb
  C^n\)}}.
\newblock {\em Indag. Math., New Ser.}, 11(1):115--127, 2000.

\bibitem{monastyrskij1974coherent}
M.~I. {Monastyrskij} and A.~M. {Perelomov}.
\newblock {Coherent states and bounded homogeneous domains}.
\newblock {\em {Rep. Math. Phys.}}, 6:1--14, 1974.

\bibitem{moore1963extensions}
C.~C. {Moore}.
\newblock {Extensions and low dimensional cohomology theory of locally compact
  groups. II}.
\newblock {\em {Trans. Am. Math. Soc.}}, 113:64--86, 1964.

\bibitem{neeb1996coherent}
K.-H. {Neeb}.
\newblock {Coherent states, holomorphic extensions, and highest weight
  representations}.
\newblock {\em {Pac. J. Math.}}, 174(2):497--542, 1996.

\bibitem{neeb1999holomorphy}
K.-H. {Neeb}.
\newblock {\em {Holomorphy and convexity in Lie theory}}, volume~28.
\newblock Berlin: de Gruyter, 1999.

\bibitem{omland2014primeness}
T.~{\AA}. {Omland}.
\newblock {Primeness and primitivity conditions for twisted group
  \(C^\ast\)-algebras}.
\newblock {\em {Math. Scand.}}, 114(2):299--319, 2014.

\bibitem{ortega-cerda2006interpolation}
J.~Ortega-Cerd{\`a}, A.~Schuster, and D.~Varolin.
\newblock Interpolation and sampling hypersurfaces for the {Bargmann}-{Fock}
  space in higher dimensions.
\newblock {\em Math. Ann.}, 335(1):79--107, 2006.

\bibitem{perelomov1986generalized}
A.~{Perelomov}.
\newblock {\em {Generalized coherent states and their applications}}.
\newblock Springer, New York, NY, 1986.

\bibitem{perelomov1972coherent}
A.~M. {Perelomov}.
\newblock {Coherent states for arbitrary Lie group}.
\newblock {\em {Commun. Math. Phys.}}, 26:222--236, 1972.

\bibitem{perelomov1974coherent}
A.~M. {Perelomov}.
\newblock {Coherent states for the Lobachevskian plane}.
\newblock {\em {Funct. Anal. Appl.}}, 7:215--222, 1974.

\bibitem{perelomov1975coherent}
A.~M. {Perelomov}.
\newblock {Coherent states and symmetric spaces}.
\newblock {\em {Commun. Math. Phys.}}, 44:197--210, 1975.

\bibitem{radulescu1998berezin}
F.~R{\u{a}}dulescu.
\newblock {\em The {{\(\Gamma\)}}-equivariant form of the {Berezin}
  quantization of the upper half plane}, volume 630 of {\em Mem. Am. Math.
  Soc.}
\newblock Providence, RI: American Mathematical Society (AMS), 1998.

\bibitem{rawnsley1977coherent}
J.~H. {Rawnsley}.
\newblock {Coherent states and K\"ahler manifolds}.
\newblock {\em {Q. J. Math., Oxf. II. Ser.}}, 28:403--415, 1977.

\bibitem{romero2020density}
J.~L. Romero and J.~T. Van~Velthoven.
\newblock The density theorem for discrete series representations restricted to
  lattices.
\newblock {\em Expo. Math.}, To Appear.

\bibitem{satake1980algebraic}
I.~{Satake}.
\newblock {Algebraic structures of symmetric domains}.
\newblock Princeton University Press., 1980.

\bibitem{seip1993beurling}
K.~{Seip}.
\newblock {Beurling type density theorems in the unit disk}.
\newblock {\em {Invent. Math.}}, 113(1):21--39, 1993.

\bibitem{seip2004interpolation}
K.~{Seip}.
\newblock {\em {Interpolation and sampling in spaces of analytic functions}},
  volume~33.
\newblock Providence, RI: American Mathematical Society (AMS), 2004.

\bibitem{stoll1977mean}
M.~{Stoll}.
\newblock {Mean value theorems for harmonic and holomorphic functions on
  bounded symmetric domains}.
\newblock {\em {J. Reine Angew. Math.}}, 290:191--198, 1977.

\bibitem{takesaki2002theory}
M.~{Takesaki}.
\newblock {\em {Theory of operator algebras I.}}
\newblock Berlin: Springer, 2002.

\bibitem{young2001introduction}
R.~M. Young.
\newblock {\em An introduction to nonharmonic {F}ourier series}.
\newblock Academic Press, Inc., San Diego, CA, first edition, 2001.

\end{thebibliography}
\end{document}